\documentclass{amsart}	

\usepackage{amsmath}
\usepackage{amssymb}
\usepackage{amsthm}
\usepackage{mathrsfs}
\usepackage[all]{xy} 
\usepackage{bm}
\usepackage{comment}
\usepackage{color}
\usepackage{manfnt}
\usepackage{indentfirst}

\usepackage{breqn}
\usepackage{autobreak}
\usepackage{tikz}
    \usetikzlibrary{arrows.meta} 

\usepackage{mathtools}
\mathtoolsset{showonlyrefs=true}

\newtheorem{Theorem}{Theorem}[section]
\newtheorem{Lemma}[Theorem]{Lemma}
\newtheorem{Proposition}[Theorem]{Proposition}
\newtheorem{Corollary}[Theorem]{Corollary}

\theoremstyle{definition}
\newtheorem{definition}[Theorem]{Definition}

\newtheorem{example}[Theorem]{Example}

\numberwithin{equation}{section}

\DeclareMathOperator{\trdeg}{\mathrm{tr\text{.}deg}}

\allowdisplaybreaks[4]

\title[MZV's with different constants]{Multiple zeta values with varying constant fields}
\author{Daichi Matsuzuki}
\email{m19044h@math.nagoya-u.ac.jp}
\address{Graduate School of Mathematics, Nagoya University, 
Furo-cho, Chikusa-ku, Nagoya, 464-8602, Japan}
\date{\today}

\begin{document}
\maketitle
\begin{abstract}
    Multiple zeta values associated with function fields with varying constant fields are dealt with simultaneously. Thakur introduced multiple zeta values in the arithmetic of positive characteristic function fields, and the definition depends on the field of constants of the chosen function field. Using Papanikolas' theory on the relationship between the $t$-motivic Galois group and the periods of a pre-$t$-motive, we show that there exist no algebraic relations which relate multiple zeta values with different constants field.
\end{abstract}
\tableofcontents
\section{Introduction}
In the arithmetic of function fields of positive characteristic, we study the polynomial ring $A=A_l:=\mathbb{F}_{p^l}[\theta]$ as an analogue of the ring $\mathbb{Z}$ of rational integers, where $p$ is a fixed prime number and $l$ is a positive integer. The rational function field $K:=\mathbb{F}_{p^l}(\theta)$ is considered to be an analogue of the field $\mathbb{Q}$ of rational numbers and the field $K_\infty:=\mathbb{F}_{p^l}((1/\theta))$ is regarded as an analogue of the field $\mathbb{R}$ of real numbers.
Our main objects are positive characteristic analogues of certain real numbers called multiple zeta values (MZV's for short):
\begin{equation}
    \zeta(s_1,\,\dots,\,s_d):=\sum_{m_1>\cdots>m_d>0}\frac{1}{m_1^{s_1}\cdots m_r^{s_r}} \in \mathbb{R}, \quad(s_1 >1,\,s_2,\,\dots,\,s_d >0).
\end{equation}
Positive characteristic analogues of real MZV's are introduced by Carlitz~(\cite{Carlitz1935}) for the case $d=1$ and by Thakur~(\cite{ThakurBook}) in general as follows:
\begin{equation}
    \zeta_l(s_1,\,\dots,\,s_d):=\sum_{ \substack{a_1,\,\dots,\,a_d \in A_{l}\\a_1,\,\dots,\,a_d\text{: monic} \\\deg a_1>\cdots>\ \deg a_d \geq 0}}\frac{1}{a_1^{s_1}\cdots a_r^{s_r}} \in \mathbb{F}_{p^l}((1/\theta)), \quad(s_1,\,\dots,\,s_r >0).
\end{equation}
(Though the notation $\zeta_A(s_1,\,\dots,\,s_d)$ or $\zeta_C(s_1,\,\dots,\,s_d)$ is usually used, we adopt the notation above in order to emphasize the dependence on $l$.)
A tuple $\mathbf{s}=(s_1,\,\dots,\,s_d)$ is called an \textit{index}. The number $\operatorname{dep}(\mathbf{s}):=d$ and the sum $\operatorname{wt}(\mathbf{s}):=s_1+\cdots+s_d$ are respectively called \textit{depth} and \textit{weight} of the presentation $\zeta_l(\mathbf{s})$.

For each fixed $l$, linear and algebraic independence of multiple zeta values $\zeta_{l}(s_1$, $\dots,\,s_r)$ are well studied by many researchers. For example, Chang~\cite{Chang2014} proved that there exists no $\overline{K}$-linear relation which relates MZV's of different weights. That is, the $\overline{K}$-linear subspace $\mathcal{Z}$ of the completion $\mathbb{C}_\infty$ of the algebraic closure $\overline{K}_\infty$ spanned by all MZV's is the direct product $\bigoplus_{j\geq 0} \mathcal{Z}_j$ of linear subspaces $\mathcal{Z}_j$ spanned by MZV's of fixed weights $j$. 
Todd~(\cite{Todd2018}) proposed a conjecture which predicts the dimension over $\overline{K}$ of each $\mathcal{Z}_j$ as an analogue of Zagier's conjecture for MZV's in characteristic $0$, and Thakur~(\cite{Thakur2017a}) gives a conjectural basis of $\mathcal{Z}_j$.
These conjectures of Todd and Thakur are solved independently in \cite{Chang2023} and \cite{Im2022}. 

Regarding algebraic independence, Chang and Yu~\cite{Chang2007} determined all algebraic relations among single zeta values (MZV's of depth $1$, also known as Carlitz zeta values). There are also Mishiba's works \cite{Mishiba2015a}, \cite{Mishiba2015}, and \cite{Mishiba2017} on algebraic independence of certain families of MZV's containing higher depth ones. We note again here that all these results concern MZV's with fixed $l$, the degree of constant field $\mathbb{F}_{p^l}$ over $\mathbb{F}_p$.

In the present paper, we consider MZV's $\zeta_{l}(s_1,\,\dots,\,s_r)$ with different $l$'s simultaneously, as they belong to the same field $\overline{\mathbb{F}_{p}}((1/\theta))$, it makes sense to consider algebraic relations of MZV's with different $l$'s. In the depth $1$ case, we already have definitive results of Chang, Papanikolas, and Yu~\cite{Chang2011}, which says that there are no algebraic relations among single zeta values of different $l$'s. Our aim is to generalize their work to higher depth setting.
The main result of this paper is the following:
\begin{Theorem}\label{main}
    Let $I$ be a finite set of indices and assume that $(s_1,\,\dots,\,s_d) \in I$ implies $(s_i,\,s_{i+1},\,\dots,\,s_j) \in I$ for any $1 \leq i \leq j \leq d$. For any distinct integers $l_1,\,\dots,\,l_r$, it holds that
    \begin{equation}\label{EqmainA}
        \operatorname{tr.deg}\overline{K}(\tilde{\pi}_{l_i},\,\zeta_{l_i}(\mathbf{s}) \mid 1 \leq i \leq r,\,\mathbf{s}\in I )=\sum_{i=1}^{r}\operatorname{tr.deg}\overline{K}(\tilde{\pi}_{l_i},\,\zeta_{l_i}(\mathbf{s}) \mid \mathbf{s}\in I ).
    \end{equation}
\end{Theorem}
\noindent See Example \ref{Omega} for the definition of the Carlitz period $\tilde{\pi}_{l}$. Note that for any finite set $I$ of indices, we can always have a finite set $\hat{I}$ containing $I$ and satisfy the assumption in the abovementioned theorem. Therefore there exists no $\overline{K}$-algebraic relation which relates relates $\zeta_{l}(\mathbf{s})$ with different $l$'s.
We can also consider the same problem for $v$-adic multiple zeta values defined by Chang and Mishiba~\cite{Chang2019a}. We hope to work on this project in the near future.

The present paper is organized as follows. We recall some theories and techniques needed for the proof of Theorem \ref{main} in Section \ref{SPreliminaries}, mainly following literature \cite{Anderson2009}, \cite{Chang2014}, and \cite{Papanikolas2008}. 
We review the notion of pre-$t$-motives, which is introduced by Papanikolas~(\cite{Papanikolas2008}) in Section \ref{SSpret}. 
In Section \ref{subsectionrat}, we recall the notion of rigid analytic trivialities of pre-$t$-motives and that Papanikolas showed that rigid analytically trivial pre-$t$-motives form a neutral Tannakian category and introduced $t$-motivic Galois group. Periods of pre-$t$-motives are also introduced. 
Section \ref{subsectionPapanikolas} is devoted to the theory of Papanikolas, which describes the relationships between $t$-motivic Galois group of given pre-$t$-motives and transcendence property of periods of the pre-$t$-motive. 
We recall the notion of $t$-motivic Carlitz multiple polylogarithms, Anderson-Thakur polynomials, and Anderson-Thakur series, which enable us to interpret MZV's as periods of pre-$t$-motives in Section \ref{subsectionATseries}. We prove Theorem \ref{main} in Section \ref{SectionGaloisGroupCalculation} by constructing suitable pre-$t$-motives which have MZV's as periods and decomposing their $t$-motivic Galois groups into a direct product of $t$-motivic Galois groups of other pre-$t$-motives.

\section{Preliminaries}\label{SPreliminaries}
We recall some theories and techniques required in the proof of Theorem \ref{main} such as pre-$t$-motives (\S\S \ref{SSpret}), rigid analytical trivialities of pre-$t$-motives (\S\S \ref{subsectionrat}), Papanikolas' theory on $t$-motivic Galois groups (\S\S\ref{subsectionPapanikolas}), and period interpretations of MZV's (\S\S \ref{subsectionATseries}). Main references are \cite{Anderson2009}, \cite{Chang2014}, and \cite{Papanikolas2008}.

Let us fix a notation. For each $l \geq1$, we define $A=A_l$ to be the polynomial ring $\mathbb{F}_{p^l}[\theta]$ and $K:=\mathbb{F}_{p^l}(\theta)$ be its fraction field. The symbol $K_\infty$ denotes the field $\mathbb{F}_p((1/\theta))$ of Laurent series, which is the completion of $K$ with respect to the absolute value given by $|\theta|_\infty=p^l$. We note that this absolute value can be extended to the algebraic closure $\overline{K_\infty}$ and we let $\mathbb{C}_\infty$ be the completion of $\overline{K_\infty}$. We take a new variable $t$ and consider the field $\mathbb{C}_\infty((t))$ of Lauren series. We define $\mathbb{T}$ to be the \textit{Tate algebra} given by
\begin{equation}
    \mathbb{T}:=\left\{ \sum_{i=m}^\infty a_i t^i \, \middle|\, m \in \mathbb{Z}, \,a_i \in \mathbb{C}_\infty, \,|a_i|_{\infty}\to 0 \text{ for $i \rightarrow \infty$}\right\}
\end{equation}
and let $\mathbb{L}$ be the field of fractions of $\mathbb{T}$. The fields $\overline{K}$, $\mathbb{C}_\infty$, and $\mathbb{L}$, and the algebra $\mathbb{T}$ are independent on the choice of $l$.
Following \cite{Anderson2004}, for an element 
\begin{equation}
    f=\sum_{i=m}^\infty a_i t^i  
\end{equation}
of $\mathbb{C}_\infty((t))$ and an integer $n$, we define its \textit{$n$-fold twist} to be 
\begin{equation}
    f^{(n)}:=\sum_{i=m}^\infty a_i^{p^n} t^i.
\end{equation}
For any matrix $B=(b_{ij})$ with $b_{ij} \in\mathbb{C}_\infty((t))$ for each $i,\,j$, we define $B^{(n)}:=(b_{ij}^{(n)})$.
A power series 
$
    \sum a_i t^i \in \mathbb{C}_\infty[[t]]
$
is said to be \textit{entire} if we have
$
    \lim_{i \to \infty}\sqrt[i]{|a_i|_\infty}=0$ and $[K_\infty(a_1,\,a_2,\,\dots):K_\infty]<\infty
$ (\cite{Anderson2004}).
An entire power series converges for any $t\in \mathbb{C}_\infty$ and we write $\mathbb{E}$ for the ring of entire power series.

\subsection{Pre-$t$-motives}\label{SSpret}
In this subsection, we recall the notion of pre-$t$-motives, which are introduced by Papanikolas~(\cite{Papanikolas2008}).
 Let $\overline{K}(t)[\sigma,\,\sigma^{-1}]$ be the non-commutative ring of Laurent polynomials over $\overline{K}(t)$ in the variable $\sigma$ subject to the reations $\sigma f = f^{-1} \sigma$ for all $f \in \overline{K}(t) $. For each $l \geq1$, we also consider the sub-ring $\overline{K}(t)[\sigma^l,\,\sigma^{-l}]$, the ring of Laurent polynomials in $\sigma^l$. We note that the center of $\overline{K}(t)[\sigma^l,\,\sigma^{-l}]$ contains $\mathbb{F}_{p^l}(t)$.

\begin{definition}[{\cite[\S 3.2]{Papanikolas2008}}]
A left $\overline{K}(t)[\sigma^l,\,\sigma^{-l}]$-module is called a \textit{pre-$t$-motive of level $l$} if it is a finite dimensional vector space over $\overline{K}(t)$. Morphisms of pre-$t$-motives are defined to be left $\overline{K}(t)[\sigma^l,\,\sigma^{-l}]$-module homomorphisms between pre-$t$-motives of level $l$.
\end{definition}

By $\sigma^l f=f^{(-l)} \sigma^l=f \sigma$, which holds for any $f \in \mathbb{F}_{p^l}(t)$, Papanikolas~\cite{Papanikolas2008} showed that the category $\mathcal{P}_l$ of level $l$ pre-$t$-motives has a structure of $\mathbb{F}_{p^l}(t)$-linear category.
He further proved in \cite[Theorem 3.2.13]{Papanikolas2008} that $\mathcal{P}_l$ is a rigid abelian tensor category over $\mathbb{F}_{p^l}(t)$ where the tensor product {\color{black}operation is given as follows. For two pre-$t$-motives $P$ and $P^\prime$ of level $l$, we define $P\otimes P^\prime:=P\otimes_{\overline{K}(t)}P^\prime$, on which $\sigma^l$ acts diagonally.}

\begin{example}[{\cite[3.2.8]{Papanikolas2008}}]
 The \textit{trivial pre-$t$-motive} denoted by $\mathbf{1}$ is the one-dimensional $\overline{K}(t)$-vector space $\overline{K}(t)$ with the $\sigma$-action given by $\sigma^l f:=f^{(-l)}$ for $f\in \overline{K}(t)$. 
\end{example}
\begin{example}[{\cite[3.2.6]{Papanikolas2008}}]
   The \textit{Carlitz motive} denoted by $C_l$ is $\overline{K}(t)$ with $\sigma^l$-action given by $\sigma f:= (t-\theta)f^{(-l)}$ for $f\in \overline{K}(t)$. For $n \geq 1$, the $n$-th tensor power of $C_l$ is denoted by $C^{\otimes n}_l:=C_l\otimes \cdots \otimes C_l$ ($n$ times). So underlying $\overline{K}(t)$-vector space of $C^{\otimes n}_l$ is also $\overline{K}(t)$ and the $\sigma^l$-action is given by $\sigma^l f:= (t-\theta)^n f^{(-l)}$ for $f\in \overline{K}(t)$.
\end{example}

Let $\{ m_1,\ldots,m_r\}$ be a fixed $\overline{K}(t)$-basis of a pre-$t$-motive $P$ of level $l$, then we have $\Phi\in \operatorname{Mat}_{r}(\overline{K}(t))$ representing the $\sigma^l$-action as $ \sigma {\bf{m}}=\Phi {\bf{m}}$ where ${\bf{m}}=(m_1,\ldots,m_r)^{tr}$. 
Throughout this paper, when we say that a pre-$t$-motive $P$ is defined by the matrix $\Phi\in \operatorname{Mat}_{r}(\overline{K}(t))$ in the situation above for convenience. 
We note that the matrix $\Phi$ defining $P$ above is invertible since $P$ is a left $\overline{K}(t)[\sigma^l,\sigma^{-l}]$-module.

We recall the techniques of derived pre-$t$-motives, which are introduced in \cite{Chang2011} and enable us to handle periods of pre-$t$-motives (see the next subsection for the definition) of different levels at the same time.
\begin{definition}[{\cite[Definition 2.2.1]{Chang2011}}]\label{Defderived}
    Let $P$ be the pre-$t$-motives of level $l$ defined by a matrix $\Phi$ in $\operatorname{GL}_r(\overline{K}(t))$. Then, its $s$-th derived pre-$t$-motive $P^{(s)}$ is defined to be the pre-$t$-motive of level $ls$ whose $\sigma^{ls}$-action is represented by 
    \begin{equation}        
    \Phi^\prime:=\Phi^{(-(s-1)l)}\Phi^{(-(s-2)l)}\cdots\Phi^{(-l)}\Phi.
    \end{equation}
\end{definition}

\subsection{Rigid analytically trivialities of pre-$t$-motives}
\label{subsectionrat}
In this subsection, we recall the notion of rigid analytical trivialities of pre-$t$-motives. 
Throughout this subsection, we consider fixed $l$ and put $q:=p^l$. Pre-$t$-motives of level $l$ are simply called pre-$t$-motives if it can cause no confusion.

For a given pre-$t$-motive $P$, we put
\begin{equation}
    P^{\text{B}}:=\{a \in \mathbb{L}\otimes_{\overline{K}(t)}P \mid \sigma^l(a)=a \}. 
\end{equation}
where $\sigma^l $-action on $\mathbb{L}\otimes_{\overline{K}(t)}P$ is given by $\sigma^l(f\otimes m):=f^{(-l)}\otimes \sigma m$ for $f \in \mathbb{L}$ and $m \in P$ (\cite{Papanikolas2008}). Note that $P^{\text{B}}$ is an $\mathbb{F}_q(t)$-vector space, and we call $P^{\text{B}}$ the \textit{Betti realization} of $P$.
If the natural map 
\begin{equation}
\mathbb{L}\otimes_{\mathbb{F}_q(t)}P^{\text{B}} \rightarrow \mathbb{L}\otimes_{\overline{K}(t)}P
\end{equation}
is an isomorphism of $\mathbb{L}$-vector spaces, then we say that $P$ is \textit{rigid analytically trivial} (\cite{Papanikolas2008}). 

Rigid analytically trivial pre-$t$-motives form a neutral Tannakian category over $\mathbb{F}_q(t)$ as the following theorem claims (for the definition of Tannakian category, we refer the readers to \cite{Deligne1982}). 
We will study periods of rigid analytically trivial pre-$t$-motives by Tannakian duality.

\begin{Theorem}[{\cite[Theorem 3.3.15]{Papanikolas2008}}]
The category $\mathcal{R}$ consisting of all rigid analytically trivial pre-$t$-motive {\color{black}forms} a neutral Tannakian category over $\mathbb{F}_q(t)$ with the fiber functor $P\mapsto P^{\text{B}}$.
\end{Theorem}

In this paper, for a rigid analytically trivial pre-$t$-motive $P$, we denote by $\langle P \rangle $ the Tannakian sub-category of $\mathcal{R}$ generated by $P$. By Tannakian duality, there exists an algebraic group $\Gamma_P$ such that $\langle P \rangle $ is equivalent to the category $\operatorname{Rep}_{\mathbb{F}_q(t)}(\Gamma_P)$ of finite dimentional linear representations of $\Gamma_P$ over $\mathbb{F}_{q}(t)$. The {\color{black}algebraic} group $\Gamma_P$ is called the \textit{$t$-motivic Galois group} of $P$ (\cite{Papanikolas2008}).

It was shown by Papanikolas~(\cite[Theorem 3.3.9]{Papanikolas2008}) that we have the following criterion for rigid analytic triviality of pre-$t$-motives. See also \cite{Anderson1986}.

\begin{Proposition}[{\cite[Theorem 3.3.9]{Papanikolas2008}}]
    Suppose that $P$ is a pre-$t$-motive of dimension $r$ over $\overline{K}(t)$ defined by $\Phi \in \operatorname{GL}_r(\overline{K}(t))$. Then $P$ is rigid analytically trivial if and only if there exists $\Psi \in { \rm GL}_r (\mathbb{L})$ such that $\Psi^{(-l)}=\Phi \Psi$.
\end{Proposition}

For a rigid analytically trivial pre-$t$-motive $P$ defined by $\Phi$, the matrix $\Psi$ in the proposition above is called a \textit{rigid analytic trivialization} of $\Phi$. {\color{black}We mention that rigid analytic trivialization of $\Phi$ is not unique. In fact,} if $\Psi$ and $\Psi^\prime$ are two rigid analytic trivializations of a matrix $\Phi$, then $\Psi^{-1}\Psi^\prime \in { \rm GL} _r(\mathbb{F}_q(t))$ (\cite[\S 4.1]{Papanikolas2008}). Let us write $\Psi^{-1}=\Theta=(\Theta_{ij})$. If an entry $\Theta_{ij}$ converges at $t=\theta$, then the value $\Theta_{ij}|_{t=\theta}$ is called a \textit{period} of $P$ (cf.~\cite{Papanikolas2008}). Because of the following proposition, the entries of the matrices $\Psi$ we consider in the following context are entire.

\begin{Proposition}[{\cite[Proposition 3.1.3]{Anderson2004}}] \label{propentire}
  Given
        $ \Phi \in \operatorname{Mat}_{r \times r}(\overline{K}[t]),
    $
    suppose that there exists  
    $\psi \in \operatorname{Mat}_{r\times 1}(\mathbb{T})$
    so that $ \psi^{(-l)}=\Phi \psi.$
    If  $\det \Phi|_{t=\theta} \neq 0$, then all entries of $\psi$ are entire.
\end{Proposition}

\begin{example}[{\cite[3.3.4]{Papanikolas2008}}]\label{Omega}
    The pre-$t$-motive $\mathbf{1}$ is defined by the matrix $(1) \in  \operatorname{GL}_1(\overline{K}(t))$, which has rigid analytic trivialization $(1) \in  \operatorname{GL}_1(\mathbb{L})$. Following \cite{Anderson2004}, we consider the infinite product
    \begin{equation}
        \Omega_l(t):= (-\theta)^{\frac{-q}{q-1}}\prod_{i \geq 1} \left( 1-\frac{t}{\theta^{q^i}} \right)\in \mathbb{C}_\infty((t)),
    \end{equation}
    where  $(-\theta)^{\frac{1}{q-1}}$ is a fixed $(q-1)$th root of $-\theta$. From the definition of $\Omega_l$, one can show that $\Omega^{(-l)}_l=(t-\theta)\Omega_l$ and so $(\Omega_l) \in  \operatorname{GL}_1(\mathbb{L})$ is a rigid analytic trivialization of the matrix $(t-\theta)\in  \operatorname{GL}_1(\overline{K}(t))$ defining {\color{black}the} Carlitz motive $C_l$.  Proposition \ref{propentire} shows that $\Omega_l$ is entire. The value  
    \begin{equation}
        \tilde{\pi}_l:=\Omega^{-1}_l|_{t=\theta}=\theta(-\theta)^{\frac{1}{q-1}}\prod_{i \geq 1}\left( 1-\frac{\theta}{\theta^{q^i}} \right)^{-1}
    \end{equation}
is a period of $C$ and is known as a fundamental period of the Carlitz module (\cite{Carlitz1935}). This value is viewed as a positive characteristic analogue of the complex period $2\pi \sqrt{-1}$ and is proven to be transcendental over $K$ by Wade~(\cite{Wade1941}) such as the classical case. 
\end{example}

Let $P$ be the pre-$t$-motives of level $l$ defined by a matrix $\Phi$ and $\Psi$ be a rigid analytic trivialization of $\Phi$. It is mentioned in \cite{Chang2011} that its $s$-th derived pre-$t$-motive $P^{(s)}$, which is the pre-$t$-motive of level $ls$ defined by $\Phi^\prime:=\Phi^{(-(s-1)l)}\Phi^{(-(s-2)l)}\cdots\Phi^{-l}$, is also rigid analytically trivial. Indeed, the defining matrix $\Phi^\prime$ has the same rigid analytically trivialization $\Psi$ as $\Phi$ since we have \begin{equation}
    \Psi^{(-ls)}=(\Psi^{-l})^{-l(s-1)}=(\Phi \Psi)^{-l(s-1)}=(\Phi^{(-l)}\Phi \Psi)^{-l(s-1)}=\cdots=\Phi^\prime \Phi.
\end{equation}

\subsection{Papanikolas' theory on $t$-motivic Galois groups}\label{subsectionPapanikolas}
This subsection is devoted to recalling Papanikolas' theory on relationships between the transcendence property of periods and $t$-motivic Galois groups. We continue to consider fixed $l \geq 1$ and to put $q:=p^l$. Pre-$t$-motives of level $l$ are simply called pre-$t$-motives if it can cause no confusion.

{\color{black}Suppose that we have $\Phi\in \operatorname{GL}_r(\overline{K}(t))$ and $\Psi\in \operatorname{GL}_r(\mathbb{L})$ for which $\Psi^{(-l)}=\Phi \Psi$.}
 We put $\Psi_1:=(\Psi_{ij}\otimes1) \in \operatorname{GL}_r(\mathbb{L}\otimes_{\overline{K}(t)}\mathbb{L})$, $\Psi_2:=(1\otimes\Psi_{ij})\in \operatorname{GL}_r(\mathbb{L}\otimes_{\overline{K}(t)}\mathbb{L})$, and \begin{equation}
 \widetilde{\Psi}:=\Psi_1^{-1}\Psi_2.
 \end{equation}
 Let us consider the algebraic sub-variety 
\begin{equation}
    \Gamma_{\Psi}:=\operatorname{Spec}\mathbb{F}_q(t)[\widetilde{\Psi}_{ij},\,1/\det\widetilde{\Psi}] \label{GammaPsi}
\end{equation}
 of $\operatorname{GL}_{r/\mathbb{F}_q(t)}$ over $\mathbb{F}_q(t)$, the smallest closed subscheme of $\operatorname{GL}_{r/\mathbb{F}_q(t)}$ which has $\widetilde{\Psi}$ as its $\mathbb{L}\otimes_{\overline{K}(t)} \mathbb{L}$-valued point. 
The following theorem is Chang's refinement of Papanikolas' theorem, which claims that the variety $\Gamma_{\Psi}$ is isomorphic to the $t$-motivic Galois group of a pre-$t$-motive $P$ if $\Psi$ is a rigid analytic trivialization {\color{black}of $\Phi$} defining $P$ and has {\color{black}connection with transcendence theory.}
This theorem plays a pivotal role in the proof of Theorem \ref{main}.
It should be mentioned that the proof of equation \eqref{Papanikolasthmtrdeg} below highly depends on the refined version (\cite[Theorem 1.2]{Chang2009}) of \textit{ABP-criterion} established in \cite{Anderson2004}.

\begin{Theorem}[{\cite[(1.1)]{Chang2009}, \cite[Theorem 5.2.2]{Papanikolas2008}}]
\label{ThmPapanikolas}
Take $\Phi$ in $\operatorname{Mat}_{ r}(\overline{K}[t])\cap\operatorname{GL}_r(\overline{K}(t))$ which has a rigid analytic trivialization $\Psi$ in $\operatorname{Mat}_{ r}(\mathbb{E})\cap\operatorname{GL}_r(\mathbb{L})$, and let $P$ be the pre-$t$-motive defined by $\Phi$. Then $\Gamma_\Psi$ defined in Equation \eqref{GammaPsi} is a geometrically smooth and geometrically connected algebraic subgroup over $\mathbb{F}_q(t)$ of $\operatorname{GL}_{r\, /\mathbb{F}_q(t)}$, and {\color{black} $\Gamma_{\Psi}$ is isomorphic to $\Gamma_{P}$ over $\mathbb{F}_q(t)$ as algebraic groups. Moreover, } if $\det \Phi |_{\theta^{1/q^i}}\neq 0$ for all $i \geq1$, we further have 
\begin{equation}
    \operatorname{trdeg}_{\overline{K}}\overline{K}\left( \Psi_{ij} |_{t=\theta}\right)= \dim \Gamma_P, \label{Papanikolasthmtrdeg}
\end{equation}
 {\color{black}where $\overline{K}\left( \Psi_{ij} |_{t=\theta}\right)$ is the field generated by all the entries of $\Psi$ evaluated at $t=\theta$ over $\overline{K}$} (refer also to \cite[Theorem 3.2]{Mishiba2017}). 
\end{Theorem}

Throughout this paper, we always identify $\Gamma_P$ with $\Gamma_\Psi$ if $\Psi$ is a rigid analytic trivialization of a matrix defining a pre-$t$-motive $P$ based on the theorem above.
Further, we simply write $\Gamma_{\Psi}$ and $\Gamma_{P}$ respectively for base changes $\Gamma_{\Psi}\times_{\operatorname{Spec}\mathbb{F}_q(t)} \operatorname{Spec}\overline{\mathbb{F}_q(t)}$ and $\Gamma_{P}\times_{\operatorname{Spec}\mathbb{F}_q(t)} \operatorname{Spec}\overline{\mathbb{F}_q(t)}$, by abuse of language.
If it causes no confusion, these symbols stand also for the group $\Gamma_{\Psi}(\overline{\mathbb{F}_q(t)})\simeq\Gamma_{P}(\overline{\mathbb{F}_q(t)})$ of $\overline{\mathbb{F}_q(t)}$-valued points in what follows.

\begin{example}
For each $l \geq 1$, the Carlitz motive $C_l$ of level $l$ is defined by the matrix $(t -\theta)$ of size $1$, which has a rigid analytic trivialization $\Psi_{C_l}=(\Omega_l)$. The $t$-motivic Galois group $\Gamma_{C_l}$ is a closed subgroup of $\mathbb{G}_{m\,/\mathbb{F}_{p^{l}}(t)}$, since $\widetilde{\Psi}_{C_l}=(\Omega_l^{-1}\otimes \Omega_l)$ is an $\mathbb{L} \otimes_{\overline{K}(t)} \mathbb{L}$-valued point of the latter group. We recall here that the period $\tilde{\pi}_l$ is transcendental over $\overline{K}$. Hence $\Gamma_{C_l}$ is $1$-dimensional and equal to $\mathbb{G}_{m\,/\mathbb{F}_{p^{l}}(t)}$. Similarly, we can show that the $t$-motivic Galois group $\Gamma_{C_l^{(s)}}$ of the $s$-the derived pre-$t$-motive is equal to $\mathbb{G}_{m\,/\mathbb{F}_{p^{ls}}(t)}$. We refer readers to \cite[Theorem 3.5.4]{Papanikolas2008}.
\end{example}

Let $P_1$ and $P_2$ be {\color{black}rigid analytically trivial} pre-$t$-motives defined respectively by the matrices $\Phi_1 \in \operatorname{GL}_{r_1}(K(t))$ and $\Phi_2 \in \operatorname{GL}_{r_2}(K(t))$. 
There exist rigid analytically trivializations $\Psi_1\in \operatorname{GL}_{r_1}(\mathbb{L})$ and $\Psi_2\in \operatorname{GL}_{r_2}(\mathbb{L})$ of these defining matrices. 
We note that the direct sum $P_1\oplus P_2$ is defined by the matrices $\Phi_1\oplus\Phi_2$ with $\Psi_1\oplus\Psi_2$ as its {\color{black} rigid analytic trivialization}. 
Here and throughout this paper, for any square matrices $B_1$ and $B_2$ the symbol $B_1\oplus B_2$ is defined to be the canonical block diagonal matrix
 \begin{equation}
    \begin{pmatrix}
        B_1& O\\
        O & B_2
    \end{pmatrix}.
\end{equation}
By the definition \eqref{GammaPsi}, the algebraic group $\Gamma_{\Psi_1\oplus\Psi_2}$ is a closed subgroup of 
\begin{equation}
    \Gamma_{\Psi_1}\times\Gamma_{\Psi_2}=\left \{ B_1\oplus B_2 \ \middle| \ B_1\in \Gamma_{\Psi_1},\,B_2 \in   \Gamma_{\Psi_2} \right\}.
\end{equation}
As the Tannakian categories $\langle P_1 \rangle $ and $\langle P_2 \rangle $ are subcategories of $\langle P_1 \oplus P_2 \rangle $, Tannakian duality yields faithfully flat morphisms $\pi_i : \Gamma_{P_1\oplus P_2}\twoheadrightarrow \Gamma_{P_i}$ of algebraic groups for $i=1,\,2$ (\cite[Proposition 2.21]{Deligne1982}). We can describe these homomorphisms in terms of inclusions $\Gamma_{P_1}=\Gamma_{\Psi_1} \subset \operatorname{GL}_{r_1},\,\Gamma_{P_2}=\Gamma_{\Psi_2} \subset \operatorname{GL}_{r_2}$, and $\Gamma_{P_1 \oplus P_2}=\Gamma_{\Psi_1 \oplus \Psi_2}  \subset \operatorname{GL}_{r_1} \times \operatorname{GL}_{r_2}$ as follows:

\begin{Lemma}\label{LemmaTannakianProj}
Take rigid analytically trivial pre-$t$-motives $P_1$ and $P_2$. In the notations as above, the following diagram commutes for $i=1,\,2$ (see \cite[Example 2.3]{Mishiba2015} for example):
\begin{equation}
\begin{tikzpicture}[auto]
\node (03) at (0, 2) {$\Gamma_{P_1\oplus P_2}$}; \node (33) at (4, 2) {$\Gamma_{P_1}\times\Gamma_{ P_2}$}; 
\node (30) at (4, 0) {$\Gamma_{P_i}$.}; 
\draw[{Hooks[right]}->] (03) to node {$\iota$}(33);
\draw[->>] (03) to node {} (30);
\draw[-] (30) to node {$\pi_i$} (03);
\draw[->>] (33) to node {$\operatorname{pr}_i$} (30);
\end{tikzpicture}
\end{equation}
\end{Lemma}

\begin{proof}
Although experts well know this lemma, we will perform a short proof here to make the present paper self-contained.
Recall that for each $\mathbb{F}_q(t)$-algebra $R$, the $\Gamma_{P_1\oplus P_2}(R)$-action on the Betti realization $R\otimes_{\mathbb{F}_q(t)}(P_1\oplus P_2)^B$ which comes from the equivalence $\langle {P_1\oplus P_2} \rangle \simeq \operatorname{Rep}_{\mathbb{F}_q(t)}(\Gamma_{P_1\oplus P_2})$ in Papanikolas theory is given by
\begin{equation}
    \Psi_{P_1\oplus P_2}^{-1}\mathbf{p}\rightarrow (\Psi_{P_1\oplus P_2}\gamma)^{-1}\mathbf{p},\quad \gamma \in  \Gamma_{P_1\oplus P_2}(R)
\end{equation}
where $\mathbf{p}$ is the $\overline{K}(t)$-basis of $P_1\oplus P_2$ corresponding to the defining matrix $\Phi_1 \oplus \Phi_2$ and the action on $\Gamma_{P_i}(R)$-action on $R\otimes_{\mathbb{F}_q(t)}P_i^B$ is given by the similar way (\cite[Theorem 4.5.3]{Papanikolas2008}). 


We compare the two tensor functors from $\operatorname{Rep}_{\mathbb{F}_q(t)}(\Gamma_{P_i})$ to $\operatorname{Rep}_{\mathbb{F}_q(t)}(\Gamma_{P_i\oplus P_2})$, respectively coming from $\pi_i $ and ${\rm pr}_i \circ\iota$. As the former category is generated by $P_i^B$, it is enough to consider the $\Gamma_{P_1\oplus P_2}$ actions on $P_i^B$ given by these functors.
The one coming from $\pi_i$ is the sub-representation of the $\Gamma_{P_1\oplus P_2}(R)$-action given above.
That coming from ${\rm pr}_i \circ\iota$ is that induced by the action $\Gamma_{P_i} \curvearrowright R\otimes_{\mathbb{F}_q(t)}P_i^B$ via the surjection ${\rm pr}_i \circ\iota:\Gamma_{P_1\oplus P_2}\rightarrow \Gamma_{P_i}$. 
We can see that these two are the same and hence the lemma follows. 
\end{proof}

Let us take a rigid analytic trivial pre-$t$-motive $P$ defined by a matrix $\Phi$ and recall the definition of $s$-th derived pre-$t$-motive $P^{(s)}$ in Definition \ref{Defderived} for $s \geq 1$. 
It was mentioned that matrix $\Psi$ is also a rigid analytic trivialization of $\Phi^\prime:=\Phi^{(-sl+l)}\Phi^{(-sl+2l)}\cdots\Phi^{(-l)}\Phi$, which defines the $s$-th derived pre-$t$-motive $P^{(s)}$ of $P$. 
The theorem of Papanikolas (Theorem \ref{ThmPapanikolas}) yields the following equalities of algebraic groups over $\overline{\mathbb{F}_{q^s}(t)}$ for any $s \geq 1$ and pre-$t$-motive $P$ whose defining matrix has rigid analytic trivialization $\Psi$:
\begin{equation} \label{derivedsameGaloisgroup}
    \Gamma_{P}=\Gamma_\Psi=\Gamma_{P^{(s)}}.
\end{equation}

\begin{example}
Let $l_1,\,\dots,\,l_r$ be distinct positive integers and $R$ be their common multiple. For each $1\leq i \leq r$, the $R/l_i$-th derived pre-$t$-motive $C_{l_i}^{(R/l_i)}$ of the Carlitz module motive $C_{l_i}$ of level $l_i$ is of level $R$ and represented by the matrix
\begin{align}
    &\left((t-\theta^{\left(-\frac{R}{l_i}l_i+l_i\right)})(t-\theta^{\left(-\frac{R}{l_i}l_i+2l_i\right)})\cdots(t-\theta^{(-l_i)})(t-\theta)\right)\\
    =&\left((t-\theta^{\left(-R+l_i\right)})(t-\theta^{\left(-R+2l_i\right)})\cdots(t-\theta^{(-l_i)})(t-\theta)\right)
\end{align}
of size $1$, which has a rigid analytical trivialization $(\Omega_{l_i})$. Hence the direct product $C_{l_1}^{(R/l_1)}\oplus \cdots \oplus C_{l_r}^{(R/l_r)}$ is of level $R$, $r$-dimensional over $\overline{K}(t)$, and represented by the diagonal matrix
\begin{equation}
    \Big((t-\theta^{\left(-R+l_1\right)})(t-\theta^{\left(-R+2l_1\right)})\cdots(t-\theta)\Big)\oplus \cdots \oplus \Big( (t-\theta^{\left(-R+l_r\right)})\cdots(t-\theta) \Big),
\end{equation}
which has a rigid analytical trivialization $\Psi:=(\Omega_{l_1})\oplus \cdots \oplus(\Omega_{l_r})$. As 
$
\Gamma_{\Psi} \simeq \Gamma_{C_{l_1}^{(R/l_1)}\oplus \cdots \oplus C_{l_r}^{(R/l_r)}}
$
is the smallest closed subscheme of $\operatorname{GL}_{r\,/\mathbb{F}_{p^R}(t)}$ such that 
$
   \widetilde{\Psi}:=(\Omega_{l_1}^{-1}\otimes\Omega_{l_1})\oplus \cdots \oplus(\Omega_{l_r}^{-1}\otimes\Omega_{l_r}) \in \Gamma_{\Psi}(\mathbb{L}\otimes_{\overline{K}(t)}\mathbb{L}),
$
this is a closed subgroup of
\begin{equation}
    \left\{\begin{pmatrix}
        a_1&0&0\\
        0&\ddots&0\\
        0&0&a_r        
    \end{pmatrix}\,\middle| \,a_1,\,\cdots,\,a_r \neq 0 \right\}\simeq \mathbb{G}_m^{r}
\end{equation}
of invertible diagonal matrices. Lemma \ref{LemmaTannakianProj} shows that the surjection $\Gamma_{C_{l_1}^{R/l_1}\oplus \cdots \oplus C_{l_r}^{R/l_r}} \twoheadrightarrow \Gamma_{C_{l_i}^{(R/l_i)}}\simeq \mathbb{G}_m$ given by the Tannakian duality coincides with the restriction of the $i$-th projection $\mathbb{G}_m^r\twoheadrightarrow \mathbb{G}_m$ to $\Gamma_{C_{l_1}^{R/l_1}\oplus \cdots \oplus C_{l_r}^{R/l_r}} \subset \mathbb{G}_m^r$ for each $i$.
\end{example}
Since Denis~(\cite{Denis1998}) proved that 
\begin{equation}
    \Omega_{l_1}^{-1}|_{t=\theta}=\tilde{\pi}_{l_1},\,\dots,\,\Omega_{l_r}^{-1}|_{t=\theta}=\tilde{\pi}_{l_r}
\end{equation}
are algebraically independent over $\overline{K}$ (see also \cite[Lemma 4.2.1]{Chang2011}) ,  we have
\begin{equation}
    \dim \Gamma_{C_{l_1}^{R/l_1}\oplus \cdots \oplus C_{l_r}^{R/l_r}}=\operatorname{tr.deg}_{\overline{K}}\overline{K}(\tilde{\pi}_{l_1},\,\dots,\,\tilde{\pi}_{l_r})=r
\end{equation}
so we have 
\begin{equation}\label{CarlitzOnly}
    \Gamma_{C_{l_1}^{R/l_1}\oplus \cdots \oplus C_{l_r}^{R/l_r}}=\mathbb{G}_m^r. 
\end{equation}

 \subsection{Carlitz multiple polylogarithms and Anderson-Thakur series}\label{subsectionATseries}
We recall the notion of $t$-motivic Carlitz multiple polylogarithms, Anderson-Thakur polynomials, and Anderson-Thakur series, which play important roles in the period interpretations of MZV's.

For an index $\mathbf{s}=(s_1,\,\dots,\,s_d) \in \mathbb{Z}_{\geq 1}^d$ and $\mathbf{u}=(u_1,\,\dots,\,u_d) \in \mathbb{C}_\infty[t]^d$, we define \textit{$t$-motivic Carlitz multiple polylogarithm} (\cite{Chang2014}) as follows:
\begin{equation}
    \mathcal{L}_{l,\,\mathbf{u},\mathbf{s}}(t):=
    \sum_{i_1>\cdots>i_d\geq0}\frac{u_1^{(i_1l)}\cdots u_d^{(i_dl)}}{((t-\theta)^{(l)}\cdots(t-\theta)^{(i_1l)})^{s_1}\cdots ((t-\theta)^{(l)}\cdots(t-\theta)^{(i_rl)})^{s_r}}
\end{equation}
Note that it satisfies the following equation:
\begin{equation}
    \mathcal{L}_{l,\,\mathbf{u},\mathbf{s}}^{(-l)}=\frac{u_d^{(-l)}\mathcal{L}_{l,\,(u_1,\,\dots,\,u_{d-1}),\,(s_1,\,\dots,\,s_{d-1})}}{(t-\theta)^{s_1+\cdots+s_{d-1}}}  \label{FrobEqPolylog} 
      +\frac{\mathcal{L}_{l,\,\mathbf{u},\mathbf{s}}}{(t-\theta)^{s_1+\cdots+s_{d}}}.
\end{equation}
In the case that $u_1,\,\dots,u_d \in \mathbb{C}_\infty$, the value at $t=\theta$ is equal to the value
\begin{equation}
    \operatorname{Li}_{l,\, \mathbf{s}}(\mathbf{u})
    :=\sum_{i_1>\cdots>i_d\geq0}\frac{u_1^{(i_1l)}\cdots u_d^{(i_dl)}}{((\theta-\theta^{(l)})\cdots(\theta-\theta^{(i_1l)}))^{s_1}\cdots ((\theta-\theta^{(l)})\cdots(\theta-\theta^{(i_rl)}))^{s_r}}.
\end{equation}
of \textit{Carlitz multiple polylogarithm} if the series in right hand side converges (\cite{Chang2014}).
For a polynomial
\begin{equation}
    u=\sum_{i=0}^m a_i t^i \in \overline{K}[t],
\end{equation}
we put $||u||_\infty:=\max_{i}(|a_i|_\infty)$. 
If we have inequalities
\begin{equation}||u_{i}||_\infty<| \theta|_\infty^{\frac{s_i
q}{q-1}}
\end{equation}
for each $1 \geq i \geq d$, then we have $\mathcal{L}_{l,\,\mathbf{u},\mathbf{s}}\in \mathbb{T}$ (\cite{Chang2014}).

Anderson and Thakur~(\cite{Anderson1990}) introduced a sequence $\mathcal{H}_{l,\,0},\,\mathcal{H}_{l,\,1},\,\dots \in A_l[t]$ of polynimials, which are called \textit{Anderson-Thakur polynomials} by the following generating series:
\begin{equation}
  \left(1-\sum_{i\geq0}\frac{\prod_{j=1}^i (t^{q^i}-\theta^{q^j})}{\prod_{j=0}^{i-1}(t^{q^i}-t^{q^j})}x^{q^i} \right)^{-1}=\sum_{s\geq0} \frac{\mathcal{H}_{l,\,s}(t)}{\Gamma_{s+1|_{\theta=t}}}x^{q^s}.
\end{equation}
Here $\Gamma_{s+1}$ is the \textit{Carlitz factorial} defined as follows: for non-negative integer $s$ with the $q$-adic digit expansion
\begin{equation}
    s=\sum_{i=0}^m s_{(i)}q^i, \quad (0 \leq s_{(i)} \leq q-1)
\end{equation}
we put 
\begin{equation}
    \Gamma_{s+1}:=\prod_{i=0}^m D_i^{s_{(i)}}\in A
\end{equation}
where $D_i$ is the product of all monic polynomial in $A$ of degree $i$, see \cite{ThakurBook} for details. Anderson-Thakur polynomials enable us to interpret MZV's in terms of special values of Carlitz multiple polylogarithms  as follows:
\begin{equation}
   \mathcal{L}_{l,\,(\mathcal{H}_{l,\,s_1-1},\,\dots,\,\mathcal{H}_{l,\,s_r-1}),\,(s_1,\,\dots,\,s_d)}(t)|_{t=\theta}=\Gamma_{s_1}\cdots\Gamma_{s_r}\zeta_{l}(s_1,\,\dots,\,s_d), \label{ATpolynom}
\end{equation}
(see \cite[Theorem 3.8.3]{Anderson1990}, \cite{Chang2009}, \cite{Chang2014}). For an index $\mathbf{s}=(s_1,\,\dots,\,s_d) \in \mathbb{Z}_{\geq 1}^d$, the \textit{Anderson-Thakur series} $\zeta^{\textrm{AT}}_{A_l}(s_1,\,\dots,\,s_r)$ is defined to be the series
\begin{equation}
    \mathcal{L}_{l,\,(\mathcal{H}_{l,\,s_1-1},\,\dots,\,\mathcal{H}_{l,\,s_r-1}),\,(s_1,\,\dots,\,s_d)}(t)\in \mathbb{C}_{\infty}((t)).
\end{equation}

Let us recall the period interpretations of multiple zeta values and {\color{black} special values} of Carlitz multiple polylogarithms {\color{black}at algebraic points} (\cite{Anderson2009} and \cite{Chang2014}). We take $\mathbf{s}=(s_1,\,\dots,\,s_r) \in \mathbb{Z}_{\geq 1}^d$ and $\mathbf{u}=(u_1,\,\dots,\,u_d) \in \overline{K}(t)^d$, and consider the pre-$t$-motive $M_l[\mathbf{u};\mathbf{s}]$ defined by
\begin{equation}
\Phi_l[\mathbf{u};\mathbf{s}] :=\begin{pmatrix}
   (t-\theta)^{s_{1}+\cdots+s_{d}} & 0& \cdots && \\
   (t-\theta)^{s_{1}+\cdots+s_{d}} u_{1}^{(-1)} & (t-\theta)^{s_{2}+\cdots+s_{d}}&0&\cdots&\\
   &\ddots&\ddots&\ddots& \\
   & &&(t-\theta)^{s_{d}}& 0\\
  & & &(t-\theta)^{s_{d}} u_{d}^{(-1)} &1
\end{pmatrix}.\label{PhiUS}
\end{equation}
As we have equation \eqref{FrobEqPolylog}, this representing matrix has rigid analytic trivialization 
\begin{align}
& \quad \ \ \Psi_l[\mathbf{u};\mathbf{s}]\label{PsiUS}
&:=\begin{pmatrix}
   \Omega^{s_{1}+\cdots+s_{d}} & 0& \cdots && \\
   \Omega^{s_{1}+\cdots+s_{d}}  L_{2,\,1} &  \Omega^{s_{2}+\cdots+s_{d}}&0&\cdots&\\
   \vdots&&\ddots&\ddots& \\
  \Omega^{s_{1}+\cdots+s_{d}}  L_{d,\,1}& &&\Omega^{s_{d}} & 0\\
  \Omega^{s_{1}+\cdots+s_{d}}  L_{d+1,\,1}&   \Omega^{s_{2}+\cdots+s_{d}} L_{d+1,\,2}&\cdots& \Omega^{s_{d}}  L_{d+1,\,d}&1
\end{pmatrix}
\end{align}
where each $L_{j,\,i}$ denotes the series $\mathcal{L}_{l,\,(u_i,\,u_{i+1},\,\dots,\,u_{j-1}),\,(s_i,\,s_{i+1},\,\dots,\,s_{j-1})}$ for $1\leq i< j\leq d+1$ and $\Omega_l$ is the series introduced in Example \ref{Omega}. Putting $u_i=\mathcal{H}_{l,\,s_i-1}$ for $1 \leq i \leq d$, we can write MZV $\zeta_{l}(s_1,\,\dots,\,s_d)$ in terms of periods of a pre-$t$-motive. 

For later use, we mentioned that the $(i,\,j)$-component of the matrix $\widetilde{\Psi[\mathbf{u};\mathbf{s}]}$ is given by
    \begin{align}
       &(\Omega^{-1}\otimes \Omega)^{s_i+\cdots+s_d}\\ &\cdot\sum_{n=j}^{i}\sum_{m=0}^{i-n}(-1)^m \sum_{\substack{n=k_0<k_1<\cdots\\
    \cdots<k_{m-1}<k_m=i}}L_{k_1,\,k_0}L_{k_2,\,k_1}\cdots L_{k_m,\,k_m-1}\otimes \Omega^{s_j+\cdots+s_{i-1}}L_{n,\,j}
    \end{align}
for $1 \leq j \leq i \leq d+1$, and $(i,\,i)$-component is given by $(\Omega^{-1}\otimes \Omega)^{s_i+\cdots+s_d}$ for $1 \leq i \leq d+1$ (see \cite{Mishiba2015}). Here, we put $L_{i,\,i}=1$ for $1 \leq i \leq d+1$ by convention. 
\section{Algebraic independence of MZV's with varying constant fields}\label{SectionGaloisGroupCalculation}

This section is devoted to a proof of Theorem \ref{main}, which says that no algebraic relations exist among MZV's with different constant fields. 
The key of the proof is Equation \eqref{mainthmeq} of algebraic groups, by which we can deduce the algebraic independence in Theorem \ref{main} using Papanikolas' theorem (Theorem \ref{ThmPapanikolas}). The one-side inclusion is given in Lemma \ref{Lemmainclusion}, and the proof of equality occupies \S\S \ref{sscal}.
For an index $\mathbf{s}=(s_1,\,s_2,\,\dots$, $s_d)\in \mathbb{Z}_{\geq 1}^d$, we define $\operatorname{dep}(\mathbf{s}):=d$, $\operatorname{wt}(\mathbf{s}):=s_1+\cdots+s_d$, and
\begin{equation}
\operatorname{Sub}(\mathbf{s}):= \{(s_i,\,s_{i+1},\,\dots,\,s_j) \mid 1\leq i \leq j \leq d \}.
\end{equation}
We take a finite set $I$ of indices and assume that the following equation holds:
\begin{equation}\label{subclosed}
    I=\bigcup_{\mathbf{s}\in I}\operatorname{Sub}(\mathbf{s}).
\end{equation}
Note that for any finite set $I$ of indices, there exists a finite set $\hat{I}$ which contains $I$ and satisfies the equation above.
We enumerate the set $I$ as $I=\{ \mathbf{s}_1,\,\dots,\,\mathbf{s}_{\# I} \}$ so that $\operatorname{dep}\mathbf{s}_{j^\prime} \leq \operatorname{dep}\mathbf{s}_{j}$ for any $1 \leq j^\prime < j \leq \#I$. If we take $1 \leq j \leq \# I$ and $\mathbf{s} \in \operatorname{Sub}(\mathbf{s}_j)$, then it holds that $\mathbf{s}=\mathbf{s}_{j^\prime}$ for some $j^\prime \leq j$. We take distinct positive integers $l_1,\,\dots,\,l_r$, and let $R$ be a common multiple of them. 

For our convenience, we fix an element $u_{i,\,s}$ of $\overline{K}[t]$ for each
$1\leq i \leq r$ and $s \geq 1$. We simply write 
\begin{align}
    
\mathcal{L}_{l_i,\,(s_1,\,\dots,\,s_{d^\prime})}&:=\mathcal{L}_{l_i,\,(u_{i,\,s_1},\,\dots,\,u_{i,\,s_{d^\prime}}),\,(s_1,\,\dots,\,s_{d^\prime})}, \\M_{l_i}[(s_1,\,\dots,\,s_{d^\prime})]&:=M_{l_i}[(u_{i,\,s_1},\,\dots,\,u_{i,\,s_{d^\prime}}),\,(s_1,\,\dots,\,s_{d^\prime})],\\\Phi_{l_i}[(s_1,\,\dots,\,s_{d^\prime})]&:=\Phi_{l_i}[(u_{i,\,s_{1}},\,\dots,u_{i,\,s_{d^\prime}}), (s_1,\,\dots,\,s_{d^\prime})], \text{ and} \\\Psi_{l_i}[(s_1,\,\dots,\,s_{d^\prime})]&:=\Psi_{l_i}[(u_{i,\,s_{1}},\,\dots,u_{i,\,s_{d^\prime}}),\,(s_1,\,\dots,\,s_{d^\prime})] 
\end{align}
for each $s_1,\,\dots,\,s_{d^\prime} \geq 1$ and $1 \leq i \leq r$.
For convergence, we suppose 
\begin{equation}\label{convergencecondition}
    ||u_{i,\,s}||_\infty<|\theta|_\infty^{\frac{s p^{l_i}}{p^{l_i}-1}}
\end{equation}
for all $1\leq i\leq r$ and $s \geq 1$.

\subsection{A simple example}

To help readers to follow the calculation for the general cases, we treat with a special case.
Let us take positive integers $m$ and $n$,
and consider the set $I:=\{ \mathbf{s}_1=(m),\,\mathbf{s}_2:=(n),\,\mathbf{s}_3:=(m,\,n)\}$ of indices.
We take up the case where $r=2$ and take distinct positive integers $l_1$ and $l_2$. Let $R$ be given their common multiple.
We put $u_{i,\,m}=H_{i,m-1}$ and $u_{i,\,n}=H_{i,n-1}$ for $i=1,\,2$ to focus on the MZV's instead of general values of Carlitz polylogarithms.

For $i=1,\,2$ and $j=0,\,1,\,2,\,3$, we write $\mathbf{M}(i,\,j)$ for the pre-$t$-motive $C_{l_i}\oplus \bigoplus_{1\leq j^\prime \leq j}M_{l_i}[\mathbf{s}_{j^\prime}]$. For example, $\mathbf{M}(1,\,0)$ is equal to $C_{l_1}$ and $\mathbf{M}(i,\,3)$ is the pre-$t$-motive of level $l_i$ defined by the matrix
\begin{align}\label{examplematrixPhi}
    \Phi(i,\,3):=(t-\theta)&\oplus \begin{pmatrix}
        (t-\theta)^{m}&\\
        (t-\theta)^{m} H_{l_i,\,m-1}^{(-l_i)}&1
    \end{pmatrix}\oplus \begin{pmatrix}
        (t-\theta)^{n}&\\
        (t-\theta)^{n} H_{l_i,\,n-1}^{(-l_i)}&1
    \end{pmatrix}\\
    &   \quad \quad\ \oplus 
    \begin{pmatrix}
        (t-\theta)^{m+n}&&\\
        (t-\theta)^{m+n} H_{l_i,\,m-1}^{(-l_i)}&(t-\theta)^{m}&\\
        &(t-\theta)^{m} H_{l_i,\,m-1}^{(-l_i)}&1
    \end{pmatrix},
\end{align}
which has a rigid analytic trivialization $\Psi(i,\,3)$ given by
\begin{equation}\label{examplematrixPsi}
        (\Omega_{l_i})\oplus 
\begin{pmatrix}
\Omega_{l_i}^{m}&0\\
\Omega_{l_i}^{m} \zeta_{l_i}^{\mathrm{AT}}(m)&1
\end{pmatrix}
\oplus
\begin{pmatrix}
\Omega_{l_i}^{n}&0\\
\Omega_{l_i}^{n} \zeta_{l_i}^{\mathrm{AT}}(n)&1
\end{pmatrix}
 \oplus
\begin{pmatrix}
\Omega_{l_i}^{m+n}&0&0\\
\Omega_{l_i}^{m+n} \zeta_{l_i}^{\mathrm{AT}}(m)&\Omega_{l_i}^{n}&0\\
\Omega_{l_i}^{m+n} \zeta_{l_i}(m,\,n) &\Omega_{l_i}^{n} \zeta_{l_i}^{\mathrm{AT}}(n)&1
\end{pmatrix}.
\end{equation}

    
\begin{example}\label{Example1}

Let us obtain the algebraic independence
\begin{equation}
\trdeg_{\overline{K}}\overline{K}\big(\tilde{\pi}_{l_1},\,\zeta_{l_1}(m),\,\zeta_{l_1}(n),\,\zeta_{l_1}(m,\,n),\,\tilde{\pi}_{l_2},\,\zeta_{l_2}(m),\,\zeta_{l_2}(n),\,\zeta_{l_2}(m,\,n)\big)=8
\end{equation}
with assuming the following algebraic independence:
\begin{align}
&\quad \trdeg_{\overline{K}}\overline{K}\big(\tilde{\pi}_{l_1},\,\zeta_{l_1}(m),\,\zeta_{l_1}(n),\,\zeta_{l_1}(m,\,n),\,\tilde{\pi}_{l_2},\,\zeta_{l_2}(m),\,\zeta_{l_2}(n)\big)\\
&=\trdeg_{\overline{K}}\overline{K}\big(\tilde{\pi}_{l_1},\,\zeta_{l_1}(m),\,\zeta_{l_1}(n),\,\tilde{\pi}_{l_2},\,\zeta_{l_2}(m),\,\zeta_{l_2}(n),\,\zeta_{l_2}(m,\,n)\big)
=7.
\end{align}

Using Papanikolas ' theory, we first interpret the desired algebraic independence in terms of $t$-motivic Galois groups.
Using the notion of derived pre-$t$-motives (Definition \ref{Defderived}), we let $M$ be the pre-$t$-motive $\mathbf{M}(1,\,3)^{(R/l_1)} \oplus \mathbf{M}(2,\,3)^{(R/l_2)}$ of level $R$, whose defining matrix has rigid analytic trivialization $\Psi=\Psi(1,\,3)\oplus \Psi(2,\,3)$.
Theorem \ref{ThmPapanikolas} gives us
\begin{align}
\dim \Gamma_{M}=\dim \Gamma_{\Psi}
&=\trdeg_{\overline{K}}\overline{K}\big(\tilde{\pi}_{l_i},\,\zeta_{l_i}(m),\,\zeta_{l_i}(n),\,\zeta_{l_i}(m,\,n)\mid i=1,\,2\big)
\end{align}
and hence it is enough to show that $\dim \Gamma_{\Psi}=8$.
If we consider the algebraic group $G$ defined by
\begin{align}
\left\{ \bigoplus_{i=1,\,2}
\begin{pmatrix} (a_i)\oplus 
\begin{pmatrix}
a_i^{m}&0\\
a_i^{m} x_i&1
\end{pmatrix}
\oplus
\begin{pmatrix}
a_i^{n}&0\\
a_i^n y_i&1
\end{pmatrix}\\
\quad \ 
\oplus
\begin{pmatrix}
a_i^{m+n}&0&0\\
a_i^{m+n} x_i&a_i^{n}&0\\
a_i^{m+n} z_i &a_i^{n}y_i&1
\end{pmatrix}
\end{pmatrix}
\, \middle| \,
a_i\neq 0 \text{ for } i=1,\,2
\right\}
\subset \operatorname{GL}_{16/\overline{\mathbb{F}_q(t)}},
\end{align}
then we have $\widetilde{\Psi} \in G(\mathbb{L} \otimes_{\overline{K}(t)} \mathbb{L})$, see Subsection \ref{subsectionPapanikolas}, and the closed immersion $\iota:\Gamma_{\Psi} \hookrightarrow G $ exists since $\Gamma_{\Psi} $ is characterized to be the smallest closed subscheme of $\operatorname{GL}_{16/\overline{\mathbb{F}_q(t)}}$ such that $\widetilde{\Psi} \in \Gamma_{\Psi}(\mathbb{L} \otimes_{\overline{K}(t)} \mathbb{L})$. 

We next define $M^\prime$ to be the pre-$t$-motive $\mathbf{M}(1,\,3)^{(R/l_1)} \oplus \mathbf{M}(2,\,2)^{(R/l_2)}$ of level $R$, whose defining matrix has a rigid analytic trivialization
\begin{align}
\Psi^\prime:=\begin{pmatrix} (\Omega_{l_1})\oplus 
\begin{pmatrix}
\Omega_{l_1}^{m}&0\\
\Omega_{l_1}^{m} \zeta_{l_1}^{\mathrm{AT}}(m)&1
\end{pmatrix}
\oplus
\begin{pmatrix}
\Omega_{l_1}^{n}&0\\
\Omega_{l_1}^{n} \zeta_{l_1}^{\mathrm{AT}}(n)&1
\end{pmatrix}\\
\quad \quad \ \oplus
\begin{pmatrix}
\Omega_{l_1}^{m+n}&0&0\\
\Omega_{l_1}^{m+n} \zeta_{l_1}^{\mathrm{AT}}(m)&\Omega_{l_1}^{n}&0\\
\Omega_{l_1}^{m+n} \zeta_{l_1}{\mathrm{AT}}(m,\,n) &\Omega_{l_1}^{n} \zeta_{l_1}^{\mathrm{AT}}(n)&1
\end{pmatrix}
\end{pmatrix}\\
\oplus
\begin{pmatrix} (\Omega_{l_2})\oplus 
\begin{pmatrix}
\Omega_{l_2}^{m}&0\\
\Omega_{l_2}^{m} \zeta_{l_2}^{\mathrm{AT}}(m)&1
\end{pmatrix}
\oplus
\begin{pmatrix}
\Omega_{l_2}^{n}&0\\
\Omega_{l_2}^{n} \zeta_{l_2}^{\mathrm{AT}}(n)&1
\end{pmatrix}
\end{pmatrix} 
\end{align}
in $\operatorname{GL}_{13/\overline{\mathbb{F}_q(t)}}(\mathbb{L})$.
By Theorem \ref{ThmPapanikolas}, the dimension of $\Gamma_{M^\prime}=\Gamma_{\Psi^\prime}$ is equal to
\begin{align}
\trdeg_{\overline{K}}\overline{K}\big(\tilde{\pi}_{l_1},\,\zeta_{l_1}(m),\,\zeta_{l_1}(n),\,\zeta_{l_1}(m,\,n),\,\tilde{\pi}_{l_2},\,\zeta_{l_2}(m),\,\zeta_{l_2}(n)\big)=7.
\end{align}

We have $\widetilde{\Psi}^\prime \in G^\prime(\mathbb{L} \otimes_{\overline{K}(t)} \mathbb{L})$ if we put
$G^\prime$ to be the algebraic subgroup of $\mathop{GL}_{13/\overline{\mathbb{F}_q(t)}}$ consists of matrices of the form
\begin{equation}
\begin{pmatrix} (a_1)\oplus 
\begin{pmatrix}
a_1^{m}&0\\
a_1^{m}x_1&1
\end{pmatrix}
\oplus
\begin{pmatrix}
a_1^{n}&0\\
a_1^{n}y_1&1
\end{pmatrix}\\
\quad \ \ \oplus
\begin{pmatrix}
a_1^{m+n}&0&0\\
a_1^{m+n}x_1&a_1^{n}&0\\
a_1^{m+n}z_1 &a_1^{n} y_1&1
\end{pmatrix}
\end{pmatrix}
\oplus
\begin{pmatrix} (a_2)\oplus 
\begin{pmatrix}
a_2^{m}&0\\
a_2^{m} x_2&1
\end{pmatrix}
\oplus
\begin{pmatrix}
a_2^{n}&0\\
a_2^{n} y_2&1
\end{pmatrix}
\end{pmatrix}
\end{equation}
Hence we have an inclusion $\Gamma_{\Psi^\prime} \subset G^\prime$ as $\Gamma_{\Psi^\prime}$ is characterized to be the smallest sub-scheme of $\mathop{GL}_{13/\overline{\mathbb{F}_q(t)}}$ which has $\widetilde{\Psi}^\prime$ as its $\mathbb{L} \otimes_{\overline{K}(t)} \mathbb{L}$-valued point.
Consequently, we have $\Gamma_{\Psi^\prime} = G^\prime$ since we have assumed that $\dim \Gamma_{\Psi^\prime}=7$ and $G^\prime$ is smooth and connected.

Similarly, we put $M^{\prime \prime}:=\mathbf{M}(1,\,2)^{(R/l_1)} \oplus \mathbf{M}(2,\,2)^{(R/l_2)}$ with the matrix
\begin{equation}
\Psi^{\prime \prime}:=
\bigoplus_{i=1,\,2}
\left( (\Omega_{l_i})\oplus 
\begin{pmatrix}
\Omega_{l_i}^{m}&0\\
\Omega_{l_i}^{m} \zeta_{l_i}^{\mathrm{AT}}(m)&1
\end{pmatrix}
\oplus
\begin{pmatrix}
\Omega_{l_i}^{n}&0\\
\Omega_{l_i}^{n} \zeta_{l_i}^{\mathrm{AT}}(n)&1
\end{pmatrix}
\right)\in \operatorname{GL}_{10/\overline{\mathbb{F}_q(t)}}(\mathbb{L})
\end{equation}
as a rigid analytic trivialization of its defining matrix.
By the similar argument to the proof of $\Gamma_{\Psi^\prime} = G^\prime$, we have $\Gamma_{M^{\prime \prime}}=\Gamma_{\Psi^{\prime \prime}} = G^{\prime \prime}$ where
\begin{align}
G^{\prime \prime}(\overline{\mathbb{F}_q(t)})&=
\left\{ \bigoplus_{i=1,\,2}
\left( (a_i)\oplus 
\begin{pmatrix}
a_i^{m}&0\\
a_i^{m} x_i&1
\end{pmatrix}
\oplus
\begin{pmatrix}
a_i^{n}&0\\
a_i^n y_i&1
\end{pmatrix}
\right)
\, \middle| \,
a_i
\neq 0 \text{ for } i=1,\,2
\right\}
\\
& \subset \operatorname{GL}_{10/\overline{\mathbb{F}_q(t)}}
\end{align}
as we have $\trdeg_{\overline{K}}\overline{K}\big(\tilde{\pi}_{l_1},\,\zeta_{l_1}(m),\,\zeta_{l_1}(n),\,\tilde{\pi}_{l_2},\,\zeta_{l_2}(m),\,\zeta_{l_2}(n)\big)
=6.$

Consider the surjections $\overline{\varphi}:G \rightarrow  \Gamma_{\Psi^\prime}=G^\prime$ and $\phi:\Gamma_{\Psi^\prime}\twoheadrightarrow \Gamma_{\Psi^{\prime \prime}}=G^{\prime \prime}$ given by
\begin{align}
&\begin{pmatrix} (a_1)\oplus 
\begin{pmatrix}
a_1^{m}&0\\
a_1^{m} x_1&1
\end{pmatrix}
\oplus
\begin{pmatrix}
a_1^{n}&0\\
a_1^n y_1&1
\end{pmatrix}\\
\quad \ 
\oplus
\begin{pmatrix}
a_1^{m+n}&0&0\\
a_1^{m+n} x_1&a_1^{n}&0\\
a_1^{m+n} z_1 &a_1^{n}y_1&1
\end{pmatrix}
\end{pmatrix}\oplus
\begin{pmatrix} (a_2)\oplus 
\begin{pmatrix}
a_2^{m}&0\\
a_2^{m} x_2&1
\end{pmatrix}
\oplus
\begin{pmatrix}
a_2^{n}&0\\
a_2^n y_2&1
\end{pmatrix}\\
\quad \ 
\oplus
\begin{pmatrix}
a_2^{m+n}&0&0\\
a_2^{m+n} x_2&a_2^{n}&0\\
a_2^{m+n} z_2 &a_2^{n}y_2&1
\end{pmatrix}
\end{pmatrix}\\
&\overset{\overline{\varphi}}{\mapsto}
\begin{pmatrix} (a_1)\oplus 
\begin{pmatrix}
a_1^{m}&0\\
a_1^{m}x_1&1
\end{pmatrix}
\oplus
\begin{pmatrix}
a_1^{n}&0\\
a_1^{n}y_1&1
\end{pmatrix}\\
\quad \ \ \oplus
\begin{pmatrix}
a_1^{m+n}&0&0\\
a_1^{m+n}x_1&a_1^{n}&0\\
a_1^{m+n}z_1 &a_1^{n} y_1&1
\end{pmatrix}
\end{pmatrix}
\oplus
\begin{pmatrix} (a_2)\oplus 
\begin{pmatrix}
a_2^{m}&0\\
a_2^{m} x_2&1
\end{pmatrix}
\oplus
\begin{pmatrix}
a_2^{n}&0\\
a_2^{n} y_2&1
\end{pmatrix}
\end{pmatrix}\\
&\overset{\phi}{\mapsto}
\left( (a_1)\oplus 
\begin{pmatrix}
a_1^{m}&0\\
a_1^{m} x_1&1
\end{pmatrix}
\oplus
\begin{pmatrix}
a_1^{n}&0\\
a_1^n y_1&1
\end{pmatrix}
\right)
\oplus
\left( (a_2)\oplus 
\begin{pmatrix}
a_2^{m}&0\\
a_2^{m} x_2&1
\end{pmatrix}
\oplus
\begin{pmatrix}
a_2^{n}&0\\
a_2^n y_2&1
\end{pmatrix}
\right).
\end{align}
Since $M^\prime$ is a direct summand of the pre-$t$-motive $M$, Tannakian duality yields a faithfully flat morphism $\varphi:\Gamma_{\Psi}\twoheadrightarrow \Gamma_{\Psi^\prime}$ as in Lemma \ref{LemmaTannakianProj}. 
It holds that
$\varphi=\overline{\varphi}\circ \iota$ by the lemma.
We also consider the morphisms $\overline{\psi}:=\phi \circ \overline{\varphi}$ and $\psi:=\phi \circ \varphi=\overline{\psi}\ \circ \iota$, see the following diagram:
\begin{center}
        \begin{tikzpicture}[auto] 
                   \node (22) at (0, 0) {$\Gamma_{\Psi}$}; \node (42) at (3, 0) {$G$}; \node (62) at (6, 0) {$\Gamma_{\Psi^\prime}$}; \node (82) at (9, 0) {$\Gamma_{\Psi^{\prime \prime}}$}; \node (822) at (9, -0.1) {};

                    \draw[{Hooks[right]}->] (22) to node {$\iota$}(42);
                    \draw[->>] (42) to node {$\overline{\varphi}$}(62);
                    \draw[->>] (62) to node {$\phi$}(82);
                    \draw[->>] (42) to [bend right=20] node {} (82);
                    \draw[-] (82) to [bend left=20] node {$\quad \quad \quad \quad \overline{\psi}=\phi \circ \overline{\varphi}$} (42);
                    \draw[->>] (22) to [bend left=20] node {$\psi = \phi \circ \varphi=\phi \circ \overline{\varphi} \circ \iota=\overline{\psi} \circ \iota$} (82);
                    \draw[white,line width=2pt] (22) to [bend right=20] node { } (62);
                    \draw[->>] (22) to [bend right=20] node {$\varphi \quad \quad \quad \quad $} (62);
                    \end{tikzpicture}.
  \end{center}
Note that morphisms $\overline{\psi}$ and $\psi$ are given as follows:

 \begin{small}
  \begin{equation}
  \bigoplus_{i=1,\,2}
\begin{pmatrix} (a_i)\oplus 
\begin{pmatrix}
a_i^{m}&0\\
a_i^{m} x_i&1
\end{pmatrix}
\oplus
\begin{pmatrix}
a_i^{n}&0\\
a_i^n y_i&1
\end{pmatrix}\\
\quad \ 
\oplus
\begin{pmatrix}
a_i^{m+n}&0&0\\
a_i^{m+n} x_i&a_i^{n}&0\\
a_i^{m+n} z_i &a_i^{n}y_i&1
\end{pmatrix}
\end{pmatrix}
\mapsto
\bigoplus_{i=1,\,2}
\left( (a_i)\oplus 
\begin{pmatrix}
a_i^{m}&0\\
a_i^{m} x_i&1
\end{pmatrix}
\oplus
\begin{pmatrix}
a_i^{n}&0\\
a_i^n y_i&1
\end{pmatrix}
\right).
  \end{equation}
  \end{small}

Using these morphisms, we put
  \begin{align}
  V&:=\ker \overline{\psi}
  =\left \{\bigoplus_{i=1,\,2} \left( I_1\oplus I_2 \oplus I_2 \oplus 
  \begin{pmatrix}
  1&0&0\\
  0&1&0\\
  z_i&0&1
  \end{pmatrix} \right) \right\} \simeq \mathbb{G}_a^2,\\
  \mathcal{V}&:= \ker \psi \subset V,  \text{ and }\\
  
  V^\prime&:=\ker \phi =\left \{ \left (I_1 \oplus I_2 \oplus I_2 \oplus 
  \begin{pmatrix}
  1&0&0\\
  0&1&0\\
  z_1&0&1
  \end{pmatrix}\right)
  \oplus \left(I_1 \oplus I_2 \oplus I_2\right) \right\}
  \simeq \mathbb{G}_a.
  \end{align}
Then the following commutative diagram shows that the composition $\mathcal{V} \overset{\iota}{\hookrightarrow} V \overset{\text{pr}_1}{\twoheadrightarrow} V^\prime$ is surjective:

\begin{center}
        \begin{tikzpicture}[auto] 
                   \node (22) at (0, 1.5) {$\mathcal{V}$}; \node (42) at (4, 1.5) {$\Gamma_{\Psi}$}; \node (62) at (8, 1.5) {$\Gamma_{\Psi^{\prime \prime}}$};
                    \node (20) at (0, 0) {$V$}; \node (40) at (4, 0) {$G$}; \node (60) at (8, 0) {$\Gamma_{\Psi^{\prime \prime}}$};
                    \node (2-2) at (0, -1.5) {$V^\prime$}; \node (4-2) at (4, -1.5) {$\Gamma_{\Psi^\prime}$}; \node (6-2) at (8, -1.5) {$\Gamma_{\Psi^{\prime \prime}}$};
                    ;

                    \draw[-] (62) to node[sloped] {$=$}(60);
                    \draw[-] (60) to node[sloped] {$=$}(6-2);
                    
                    \draw[->>] (40) to node {$\overline{\varphi}$}(4-2);

                    \draw[{Hooks[right]}->] (42) to node {$\iota$}(40);
                    \draw[{Hooks[right]}->] (22) to node {}(20);

                    \draw[{Hooks[right]}->] (22) to node {}(42);
                    \draw[{Hooks[right]}->] (20) to node {}(40);
                    \draw[{Hooks[right]}->] (2-2) to node {}(4-2);
    
                    \draw[->>] (42) to node {$\psi$}(62);
                    \draw[->>] (40) to node {$\overline{\psi}$}(60);
                    \draw[->>] (4-2) to node {$\phi$}(6-2);

                    \draw[->>] (20) to node {$\text{pr}_1$}(2-2);

                    \draw[white,line width=6pt] (42) to [bend right=50] node { } (4-2);
                    \draw[->>] (42) to [bend right=50] node {} (4-2);
                    \draw[-] (4-2) to [bend left=50] node {$\begin{matrix}\\ \\ \varphi=\overline{\varphi}\circ \iota\end{matrix}$} (42);
                \end{tikzpicture}.
    \end{center}

The surjectivity of the composition $\mathcal{V} \overset{\iota}{\hookrightarrow} V \overset{\text{pr}_1}{\twoheadrightarrow} V^\prime$ gives us $z_R \in \overline{\mathbb{F}_q(t)}$ such that
\begin{equation}
R:=\left( I_1 \oplus I_2 \oplus I_2 \oplus 
\begin{pmatrix}
1&0&0\\
0&1&0\\
1&0&1
\end{pmatrix} \right)
\oplus 
\left( I_1 \oplus I_2 \oplus I_2 \oplus 
\begin{pmatrix}
1&0&0\\
0&1&0\\
z_R&0&1
\end{pmatrix} \right) 
\end{equation}
is an element of the group $\mathcal{V}(\overline{\mathbb{F}_q(t)})$.
If we take any $\alpha \in \overline{\mathbb{F}_q(t)}^{\times}$ and consider the matrix 
\begin{equation}
Q^\prime:=\begin{pmatrix} (\alpha)\oplus 
\begin{pmatrix}
\alpha^{m}&0\\
0&1
\end{pmatrix}
\oplus
\begin{pmatrix}
\alpha^{n}&0\\
0&1
\end{pmatrix}\\
\quad \ \ \oplus
\begin{pmatrix}
\alpha^{m+n}&0&0\\
0&\alpha^{n}&0\\
0 &0&1
\end{pmatrix}
\end{pmatrix}
\oplus
\begin{pmatrix} I_1\oplus 
I_2
\oplus
I_2
\end{pmatrix}
\end{equation}
in $\Gamma_{\Psi^\prime}(\overline{\mathbb{F}_q(t)})$, then surjectivity of $\varphi$ enable us to pick
\begin{equation}
Q:=\begin{pmatrix} (\alpha)\oplus 
\begin{pmatrix}
\alpha^{m}&0\\
0&1
\end{pmatrix}
\oplus
\begin{pmatrix}
\alpha^{n}&0\\
0&1
\end{pmatrix}\\
\quad \ 
\oplus
\begin{pmatrix}
\alpha^{m+n}&0&0\\
0&\alpha^{n}&0\\
0 &0&1
\end{pmatrix}
\end{pmatrix}\oplus
\begin{pmatrix} I_1\oplus 
I_2
\oplus
I_2\\
\oplus
\begin{pmatrix}
1&0&0\\
0&1&0\\
z_Q&0&1
\end{pmatrix}
\end{pmatrix}\in \Gamma_{\Psi}(\overline{\mathbb{F}_q(t)}),
\end{equation}
where $z_Q$ is some element in $\overline{\mathbb{F}_q(t)}$.

As the algebraic subgroup $\mathcal{V}$ is codimension $0$ or $1$ in $V \simeq \mathbb{G}_a^2$, we have polynomials $P_1(X)$ and $P_2(X)$ such that the algebraic set $\mathcal{V}(\overline{\mathbb{F}_q(t)})$ is given as follows:
\begin{equation}
\mathcal{V}(\overline{\mathbb{F}_q(t)})=\left \{\bigoplus_{i=1,\,2} \left( I_1\oplus I_2 \oplus I_2 \oplus 
  \begin{pmatrix}
  1&0&0\\
  0&1&0\\
  z_i&0&1
  \end{pmatrix} \right) \, \middle|\, P_1(z_1)+P_2(z_2)=0\right\},
\end{equation}
see \cite[Corollary 1.8]{Conrad2015}.
On the other hand, the commutator $Q^{-1}R Q$ is equal to
\begin{equation}\label{Examplecommutator}
\left( I_1\oplus 
I_2
\oplus
I_2
\oplus
\begin{pmatrix}
1&0&0\\
0&1&0\\
\alpha^{m+n}&0&1
\end{pmatrix}
\right)
\oplus
\left( I_1\oplus 
I_2
\oplus
I_2
\oplus
\begin{pmatrix}
1&0&0\\
0&1&0\\
z&0&1
\end{pmatrix}
\right) 
\end{equation}
and is in $\mathcal{V}(\overline{\mathbb{F}_q(t)})$ as $\mathcal{V}$ is normal in $\Gamma_{\Psi}$.
Hence we have $P_1(\alpha^{m+n})+P_2(z)=0$.

As $\alpha$ is an arbitrarily chosen element of the infinite set $ \overline{\mathbb{F}_q(t)}^\times$, we have $P_1(X)=0\in \overline{\mathbb{F}_q(t)}[X]$, and we can also show $P_2(X)=0$ by the similar arguments.
Consequently, we have $\mathcal{V}=V$ and $\dim \mathcal{V}=2$.
By the exact sequence $1 \rightarrow \mathcal{V} \overset{\iota}{\hookrightarrow}\Gamma_{\Psi} \overset{\phi}{\twoheadrightarrow} \Gamma_{\Psi^{\prime \prime}}=G^{\prime \prime} \twoheadrightarrow 1$, we have $\dim \Gamma_{\Psi}=8$ and hence the equalities
\begin{equation}
\trdeg_{\overline{K}}\overline{K}\big(\tilde{\pi}_{l_i},\,\zeta_{l_i}(m),\,\zeta_{l_i}(n),\,\zeta_{l_i}(m,\,n)\mid i=1,\,2\big)
=\dim \Gamma_{\Psi}=8.
\end{equation}
\end{example}
\subsection{Constructions of certain pre-$t$-motives}\label{subsectionconstructionpret}

\begin{definition}\label{mainpret}
We take a finite set $I$ of indices which satisfies Equation \eqref{subclosed} and fix its enumeration $I=\{ \mathbf{s}_1,\,\dots,\,\mathbf{s}_{\# I} \}$ so that $\operatorname{dep}\mathbf{s}_{j^\prime} \leq \operatorname{dep}\mathbf{s}_{j}$ for any $1 \leq j^\prime < j \leq \#I$.
Let $l_1,\,\dots,\,l_r$ be distinct positive integers and $R$ be their common multiple.
We also take an element $u_{i,\,s}$ of $\overline{K}[t]$ for each
$1\leq i \leq r$ and $s \geq 1$ such that the inequalities \eqref{convergencecondition} hold.

For $1\leq i \leq r$ and $0 \leq j \leq \#I$, we put $M(i,\,j)$ to be the $R/l_i$-th derived pre-$t$-motive of $\mathbf{M}(i,\,j):=C_{l_i} \oplus\bigoplus_{j^\prime \leq j}M_{l_i}[\mathbf{s}_{j^\prime}]$.
We note that $M(i,\,0)=C_{l_i}^{(R/l_i)}$. 
\end{definition}

\begin{example}
Let us take positive integers $m$ and $n$,
and consider the set $I:=\{ \mathbf{s}_1=(m),\,\mathbf{s}_2:=(n),\,\mathbf{s}_3:=(m,\,n)\}$ of indices.
We fix $1 \leq i \leq r $ and put $u_{i,\,m}=H_{i,m-1}$ and $u_{i,\,n}=H_{i,n-1}$. 
The pre-$t$-motive $\mathbf{M}(i,\,3):=C_{l_i} \oplus\bigoplus_{j^\prime \leq j}M_{l_i}[\mathbf{s}_{j^\prime}]$ is that of level $l_i$ defined by the matrix $\Phi=\Phi(i,\,3) \in \operatorname{GL}_{8/\overline{\mathbb{F}_ q(t)}}(\overline{K}[t])$ in Equation \eqref{examplematrixPhi}, which has a rigid analytic trivialization $\Psi(i,\,3)$ $ \in \operatorname{GL}_{8/\overline{\mathbb{F}_ q(t)}}(\mathbb{L})$ given by in Equation \eqref{examplematrixPsi}.
The pre-$t$-motive $M(i,\,3)$ is of level $R$ and defined by the matrix $$\Phi^{(-R+l_i)}\Phi^{(-R+2l_i)}\cdots \Phi^{(-l_i)}\Phi,$$ which also has $\Psi(i,\,3)$ as its rigid analytic trivialization.
\end{example}

For any index $\mathbf{s}$ and $1 \leq j \leq \# I$ with $\mathbf{s} \in \operatorname{Sub}(\mathbf{s}_j)$, it follows that $\mathbf{s}=\mathbf{s}_{j^\prime}$ for some $j^\prime \leq j$ by assumptions on the set $I$ and on its enumeration. Thus we have
\begin{equation}
    \dim \Gamma_{M(i,\,j)}=\trdeg_{\overline{K}}\overline{K}\Big(\mathcal{L}_{l_i,\,\mathbf{s}_{j^\prime}}|_{t=\theta},\,\Omega_{l_i}|_{t=\theta} \ \Big|\  1 \leq j^\prime \leq j \Big)
\end{equation}
for each $1 \leq i \leq r$ and $0 \leq j \leq \# I$ by Theorem \ref{ThmPapanikolas} and Equations \eqref{derivedsameGaloisgroup} and \eqref{PsiUS}, hence  inequalities 
\begin{align}\label{successor}
\dim \Gamma_{M(i,\,j-1)} \leq \dim \Gamma_{M(i,\,j)} \leq \dim \Gamma_{M(i,\,j-1)}+1
\end{align}
hold for $1 \leq j \leq \# I$.

\begin{Lemma}\label{LemmaTannakianProjfixedl}
    We have a faithfully flat morphism 
    \begin{equation}\label{Tannakianproj}
 \Gamma_{M(i,\,j)} \twoheadrightarrow \Gamma_{M(i,\,j^\prime)}
    \end{equation}
    by Tannakian duality for $1 \leq i \leq r$ and $0 \leq j^\prime \leq j \leq \#I$.
\end{Lemma}

\begin{proof}
As the pre-$t$-motives $M(i,\,j^\prime)$ is a direct summand of $M(i,\,j)$ for each $0 \leq j^\prime \leq j \leq \#I$, Lemma \ref{LemmaTannakianProj} yields the assertion.
   \end{proof}
We will consider a concrete example of these faithfully flat morphisms in Example \ref{ExampleGammahatmorph}.
Using these morphisms, we define algebraic groups $U_{i,\,j}$ and $V_{i,\,j}$ as follows:
\begin{definition}\label{DefUandV}
For $1 \leq i \leq r$ and $0  \leq j \leq \#I$, we put a $U_{i,\,j}$ to be the kernel of the morphism $\Gamma_{M(i,\,j)}\twoheadrightarrow \Gamma_{C_{l_i}^{(R/l_i)}}=\Gamma_{M(i,\,0)}$ given in Lemma \ref{LemmaTannakianProjfixedl}. 
If $j \geq 1$, we put the $V_{i,\,j}$ to be the kernel of the restriction $U_{i,\,j}\twoheadrightarrow U_{i,\,j-1}$ induced by the surjection $\Gamma_{M(i,\,j)}\twoheadrightarrow\Gamma_{M(i,\,j-1)}$ given by the lemma, see the diagram below:

\begin{center}
\begin{tikzpicture}[auto]
 \node (24) at (2, 1.5) {$V_{i,\,j}$}; \node (44) at (4, 1.5) {$U_{i,\,j}$}; \node (64) at (7.5, 1.5) {$U_{i,\,j-1}$};
\node (42) at (4, 0) {$\Gamma_{M(i,\,j)}$}; \node (62) at (7.5, 0) {$\Gamma_{M(i,\,j-1)}$};
\node (0) at (4, -1.5) {$\Gamma_{M(i,\,0)}$}; \node (00) at (7.5, -1.5) {$\Gamma_{M(i,\,0)}$};


\draw[{Hooks[right]}->] (64) to node[sloped] {}(62);
\draw[->>] (42) to node {}(62);
\draw[->>] (44) to node {}(64);
\draw[-](0) to node{$=$} (00);
\draw[->>] (42) to node {}(0);
\draw[->>] (62) to node {}(00);

\draw[{Hooks[right]}->] (44) to node {}(42);

\draw[{Hooks[right]}->] (24) to node {}(44);



\end{tikzpicture}.
\end{center}
\end{definition}

For example, $\Gamma_{M(i,\,0)}$ is defined to be the $t$-motivic Galois group of the $R/l_i$-th derived pre-$t$-motive $C_{l_i}^{(R/l_i)}$ of the Carlitz motive $C_{l_i}$ and $U_{i,\,0}$ is the trivial algebraic group. 
\subsection{The varieties $\hat{\Gamma}_{j}$ containing algebraic groups $\Gamma_{M(i,\,j)}$}\label{eachlevel}
This subsection discusses $t$-motivic Galois groups of pre-$t$-motives introduced in Definition \ref{mainpret}. Our aim in this subsection is to show that the subgroup $V_{i,\,j}\subset \Gamma_{M(i,\,j)}$ introduced in Definition \ref{DefUandV} can be regarded as an algebraic subgroup of $\mathbb{G}_a$ for any $1 \leq i \leq r$ and $1 \leq j \leq \#I $. In order to do that, we construct an explicit algebraic variety $\hat{\Gamma}_j$ over the algebraically closed field $\overline{\mathbb{F}_q(t)}$ for each $0 \leq j \leq \# I$ and observe that the $t$-motivic Galois group $\Gamma_{M(i,\,j)}$ is a closed subscheme of $\hat{\Gamma}_j$ for each $i$. 


 We take variables $a$ and $x_{\mathbf{s}}$ for each index $\mathbf{s}\in I$.
For $\mathbf{s}=(s_1,\,\dots,\,s_d)\in I$, we write $X_{\mathbf{s}}$ for the lower triangle square matrix 
\begin{equation}
    \begin{pmatrix}
        a^{s_1+\cdots+s_d}&0&\cdots&\cdots&0\\
        a^{s_1+\cdots+s_d} x_{(s_1)}& a^{s_2+\cdots+s_d}&0&\cdots&0\\
        a^{s_1+\cdots+s_d} x_{(s_1,\,s_2)}&a^{s_2+\cdots+s_d} x_{(s_2)} & \ddots&\vdots&\vdots\\
        \vdots& \vdots& \ddots& a^{s_d}&0\\
        a^{s_1+\cdots+s_d} x_{(s_1,\,\dots,\,s_d)}&a^{s_2+\cdots+s_d} x_{(s_2,\,\dots,\,s_d)}&\cdots&a^{s_d}x_{(s_d)} &1
    \end{pmatrix}.
\end{equation}
\begin{definition}
Take a finite set $I$ of indices which satisfies Equation \eqref{subclosed} and its enumeration $I=\{ \mathbf{s}_1,\,\dots,\,\mathbf{s}_{\# I} \}$ such that $\operatorname{dep}\mathbf{s}_{j^\prime} \leq \operatorname{dep}\mathbf{s}_{j}$ for any $1 \leq j^\prime < j \leq \#I$.  Define $\hat{\Gamma}_j$ to be the closed subscheme of $\operatorname{GL}_{N_j\,/\mathbb{F}_{p^R}(t)}$, where $N_j:=1+(\operatorname{dep}(\mathbf{s}_1)+1)+\cdots+(\operatorname{dep}(\mathbf{s}_j)+1)$, consisting of matrix of the form $(a)\oplus X_{\mathbf{s}_1}\oplus\cdots\oplus X_{\mathbf{s}_j}$ with $a \neq 0$ for $0 \leq j \leq \#I$.
\end{definition}
The algebraic set $\hat{\Gamma}_j$ is isomorphic to the smooth and irrreducible variety
\begin{equation}
    \mathbb{G}_m \times \mathbb{A}^j=\operatorname{Spec} \mathbb{F}_{p^R}(t)[a,\,a^{-1},\,x_{\mathbf{s}_1},\,\dots,\,x_{\mathbf{s}_j}].
\end{equation}
We have the first projection $\operatorname{pr}_j:\hat{\Gamma}_j \twoheadrightarrow \mathbb{G}_m$ given by $(a)\oplus X_{\mathbf{s}_1}\oplus\cdots\oplus X_{\mathbf{s}_j} \mapsto (a)$ and we define $\hat{U}_j$ to be the inverse image of the unit element of $\mathbb{G}_m$.

\begin{example}\label{ExampleGammahat}
We take positive integers $m$ and $n$,
and consider the set $I:=\{ \mathbf{s}_1=(m),\,\mathbf{s}_2:=(n),\,\mathbf{s}_3:=(m,\,n)\}$ of indices again.
Then the algebraic sets $\hat{\Gamma}_2$ and
$\hat{\Gamma}_3$ are given as follows:
\begin{align}
\hat{\Gamma}_2&=\left\{    (a)\oplus 
\begin{pmatrix}
a^{m}&0\\
a^{m}x_{\mathbf{s}_1}&1
\end{pmatrix}
\oplus
\begin{pmatrix}
a^{n}&0\\
a^{n}x_{\mathbf{s}_2}&1
\end{pmatrix}
\, \middle| \, a\neq 0\right\}\\
\hat{\Gamma}_3&=\left\{    (a)\oplus 
\begin{pmatrix}
a^{m}&0\\
a^{m}x_{\mathbf{s}_1}&1
\end{pmatrix}
\oplus
\begin{pmatrix}
a^{n}&0\\
a^{n}x_{\mathbf{s}_2}&1
\end{pmatrix}\oplus
\begin{pmatrix}
a^{m+n}&0&0\\
a^{m+n}x_{\mathbf{s}_1}&a^{n}&0\\
a^{m+n}x_{\mathbf{s}_3} &a^{n} x_{\mathbf{s}_2}&1
\end{pmatrix}
\, \middle| \, a\neq 0\right\}.
\end{align}
\end{example}

The natural projection $\hat{\Gamma}_j \twoheadrightarrow \hat{\Gamma}_{j-1}$ given by
\begin{equation}
    (a)\oplus X_{\mathbf{s}_1}\oplus\cdots\oplus X_{\mathbf{s}_j} \mapsto  (a)\oplus X_{\mathbf{s}_1}\oplus\cdots\oplus X_{\mathbf{s}_{j-1}}
\end{equation}
restricts to the surjection $\operatorname{pr}_j^\prime:\hat{U}_j \twoheadrightarrow\hat{U}_{j-1}$ for each $1 \leq j \leq \#I$ and we write $\hat{V}_j \subset \hat{U}_j$ for the inverse image of the unit element. By definition, we have the following equality
\begin{small}
\begin{equation}\label{shapeVhat}
   \hat{V}_j =\left \{(1)\oplus I_{\operatorname{dep}(\mathbf{s}_1)+1} \oplus \cdots \oplus I_{\operatorname{dep}(\mathbf{s}_{j-1})+1} \oplus 
    \begin{pmatrix}
        1&0 & \cdots&\cdots& 0 \\
        0& 1&0&\cdots&\vdots\\
        \vdots&0&\ddots&\ddots&\vdots\\
        0&\vdots&\ddots&1&0\\
        x_{\mathbf{s}_j}& 0&\cdots&0&1
    \end{pmatrix}\,\middle | \, x_{\mathbf{s}_j} \in \mathbb{G}_a \right\}.
\end{equation}
\end{small}
Therefore $\hat{V}_j$ is an algebraic subgroup of $\operatorname{GL}_{N_j\,/\mathbb{F}_{p^R}(t)}$ and is isomorphic to $\mathbb{G}_a$.

    Actually, the subscheme $\hat{\Gamma}_j$ is an algebraic subgroup of $\operatorname{GL}_{N_j\,/\overline{\mathbb{F}_{p^R}(t)}}$ and $\operatorname{pr}_j$, $\operatorname{pr}_j^\prime$ are group homomorphisms for each $j$. Hence, it turns that $\hat{U}_j$ and $\hat{V}_j$ are kernels of these homomorphisms. However, we omit to confirm that $\hat{\Gamma}_j$ is a subgroup since it is unnecessary for our purpose.

\begin{Lemma}\label{lemmaexolicit}
    Take $1 \leq i \leq r$ and a finite set $I$ of indices which satisfies Equation \eqref{subclosed},
and fix its enumeration $I=\{ \mathbf{s}_1,\,\dots,\,\mathbf{s}_{\# I} \}$ so that $\operatorname{dep}\mathbf{s}_{j^\prime} \leq \operatorname{dep}\mathbf{s}_{j}$ for any $1 \leq j^\prime < j \leq \#I$.
Take also an element $u_{i,\,s}$ of $\overline{K}[t]$ such that
\begin{equation}
    ||u_{i,\,s}||_\infty<|\theta|_\infty^{\frac{s p^{l_i}}{p^{l_i}-1}}
\end{equation}
for each $s \geq 1$.
Then the $t$-motivic Galois group $\Gamma_{M(i,\,j)}$ of $M(i,\,j)$ in Definition \ref{mainpret} is a closed subscheme of $\hat{\Gamma}_{j}$ for each $1 \leq i \leq r$ and $0 \leq j \leq \# I$. 
\end{Lemma}
Note that the varieties $\hat{\Gamma}_j$, $\hat{U}_j$, and $\hat{V}_j$ are independent of $1 \leq i \leq r$ and of the choice of the sequence $u_{i,\,1},\,u_{i,\,2},\,\dots \in \overline{K}[t]$.
\begin{proof}
    We put $\Psi:=(\Omega_{l_i})\otimes \Psi_{l_i}[\mathbf{s}_1]\oplus\cdots \oplus \Psi_{l_i}[\mathbf{s}_j]$. 
    Recalling the calculation at the end of Subsection \ref{subsectionATseries}, 
we can see that the matrix
\begin{equation}
    \widetilde{\Psi}:=\widetilde{\left((\Omega_{l_i})\otimes \Psi_{l_i}[\mathbf{s}_1]\oplus\cdots \oplus \Psi_{l_i}[\mathbf{s}_j]\right)}=\widetilde{(\Omega_{l_i})}\otimes\widetilde{\Psi}_{l_i}[\mathbf{s}_1]\oplus \cdots \oplus \widetilde{\Psi}_{l_i}[\mathbf{s}_j] 
\end{equation}
is an element of $\hat{\Gamma}_j(\mathbb{L}\otimes_{\overline{K}(t)} \mathbb{L})$; the variable $a$ corresponds to $\Omega_{l_i}^{-1}\otimes \Omega_{l_i}$ and the variable $x_{\mathbf{s}}$ corresponds to
\begin{equation}
    \sum_{n=1}^{d+1}\sum_{m=0}^{d+1-n}(-1)^m \sum_{\substack{n=k_0<k_1<\cdots\\
    \cdots<k_{m-1}<k_m=d+1}}L_{k_1,\,k_0}L_{k_2,\,k_1}\cdots L_{k_m,\,k_m-1}\otimes \Omega^{s_1+\cdots+s_d}L_{n,\,1}
\end{equation}
for each $\mathbf{s}=(s_1,\,\dots,\,s_d) \in I$, where $L_{j^\prime,\,j}$ is the series
\begin{equation}
    \mathcal{L}_{l_i,\,(s_{j},\,\dots,\,s_{j^\prime-1})}
\end{equation}
for $1 \leq j \leq j^\prime \leq d+1$. So the group $\Gamma_{\Psi}$ is a closed subscheme of $\hat{\Gamma}_j$ as $\Gamma_{\Psi}$ was characterized as the smallest closed subscheme of $\operatorname{GL}_{N_j}$ which has $\widetilde{\Psi}$ as its $\mathbb{L}\otimes_{\overline{K}(t)}\mathbb{L}$-valued point. 
Theorem \ref{ThmPapanikolas} enables us to identify the $t$-motivic Galois group $\Gamma_{M(i,\,j)}$ with $\Gamma_{\Psi}$ as algebraic groups over $\overline{\mathbb{F}_{p^R}(t)}$.
\end{proof}
For $1 \leq i \leq r$ and $0 \leq j^\prime< j \leq \# I$, Lemma \ref{LemmaTannakianProj} yields the commutative diagram
\begin{equation}\label{wholespacediagram}
\begin{tikzpicture}[auto]
\node (03) at (0, 2) {$\Gamma_{M(i,\,j)}$}; \node (33) at (3, 2) {$\Gamma_{M(i,\,j^\prime)}$}; 
\node (00) at (0, 0) {$\hat{\Gamma}_{j}$}; \node (30) at (3, 0) {$\hat{\Gamma}_{j^\prime}$,}; 
\draw[->>] (03) to node {$\phi$}(33);
\draw[->>] (00) to node {projection}(30);
\draw[{Hooks[right]}->] (03) to node {$\iota$} (00);
\draw[{Hooks[right]}->] (33) to node {$\iota$} (30);
\end{tikzpicture}
\end{equation}
where $\phi$ is the morphism given in Lemma \ref{LemmaTannakianProjfixedl} and the vertical lines are closed immersions given by Lemma \ref{lemmaexolicit}.

\begin{example}\label{ExampleGammahatmorph}
We continue to consider the case where positive integers $m$ and $n$ are took and the set $I:=\{ \mathbf{s}_1=(m),\,\mathbf{s}_2:=(n),\,\mathbf{s}_3:=(m,\,n)\}$ of indices.
In this case, we have inclusions $\Gamma_{M(i,\,2)} \subset \hat{\Gamma}_2$ and
$\Gamma_{M(i,\,3)} \subset \hat{\Gamma}_3$ for each $i$ (see Example \ref{ExampleGammahat}), and we can describe the faithfully flat morphism in Lemma \ref{LemmaTannakianProjfixedl} as follows
\begin{align}
&\left(    (a)\oplus 
\begin{pmatrix}
a^{m}&0\\
a^{m}x_{\mathbf{s}_1}&1
\end{pmatrix}
\oplus
\begin{pmatrix}
a^{n}&0\\
a^{n}x_{\mathbf{s}_2}&1
\end{pmatrix}\oplus
\begin{pmatrix}
a^{m+n}&0&0\\
a^{m+n}x_{\mathbf{s}_1}&a^{n}&0\\
a^{m+n}x_{\mathbf{s}_3}&a^{n} x_{\mathbf{s}_2}&1
\end{pmatrix}
\right)\\
\mapsto&\left(    (a)\oplus 
\begin{pmatrix}
a^{m}&0\\
a^{m}x_{\mathbf{s}_1}&1
\end{pmatrix}
\oplus
\begin{pmatrix}
a^{n}&0\\
a^{n}x_{\mathbf{s}_2}&1
\end{pmatrix}
\right).
\end{align}
\end{example}

Lemma \ref{lemmaexolicit} above shows that the group scheme $V_{i,\,j}$ (see Definition \ref{DefUandV}) can be seen as an algebraic subgroup of $\mathbb{G}_a$:
\begin{Corollary}\label{VinGa}
    For $1 \leq i \leq r$ and $1 \leq  j \leq \# I$, the algebraic group $V_{i,\,j}$ can be regarded as a subgroup of $\mathbb{G}_a$.
\end{Corollary}
\begin{proof}
By considering Diagram \eqref{wholespacediagram} with $i^\prime=i-1$, we note that the algebraic group $V_{i,\,j}$ is a closed subgroup of $\hat{V}_j$ for each $i,\,j$. Now, the lemma follows from Equation \eqref{shapeVhat}.
\end{proof}

\subsection{The direct sums of the pre-$t$-motives}
The key step of the proof of Theorem \ref{main} is to get the equation $\Gamma_{M(1,\,j_1)\oplus M(2,\,j_2)\oplus \cdots \oplus M(r,\,j_r)} =\Gamma_{M(1,\,j_1)}\times \Gamma_{M(2,\,j_2)} \times \cdots \times \Gamma_{M(r,\,j_r)}$ for $0 \leq j_1,\,\dots,\,j_r \leq \#I$. 
We first verify the following inclusion:

\begin{Lemma}\label{Lemmainclusion}
Let us take a finite set $I$ of indices which satisfies Equation \eqref{subclosed}
and its enumeration $I=\{ \mathbf{s}_1,\,\dots,\,\mathbf{s}_{\# I} \}$ such that $\operatorname{dep}\mathbf{s}_{j^\prime} \leq \operatorname{dep}\mathbf{s}_{j}$ for any $1 \leq j^\prime < j \leq \#I$.
Take an element $u_{i,\,s}$ of $\overline{K}[t]$ which satisfies the inequality 
\begin{equation}
    ||u_{i,\,s}||_\infty<|\theta|_\infty^{\frac{s p^{l_i}}{p^{l_i}-1}}
\end{equation}
for each
$1\leq i \leq r$ and $s \geq 1$.
For any choice of $0 \leq j_1,\,\dots,\,j_r \leq \#I$, there exists a closed immersion
\begin{equation}
    \Gamma_{M(1,\,j_1)\oplus M(2,\,j_2)\oplus \cdots \oplus M(r,\,j_r)} \subset\Gamma_{M(1,\,j_1)}\times \Gamma_{M(2,\,j_2)} \times \cdots \times \Gamma_{M(r,\,j_r)} \label{mainimmersion}
\end{equation}
of algebraic groups where $M(i,\,j)$ is the pre-$t$-motive in Definition \ref{mainpret} for each $1 \leq i \leq r $ and $0 \leq j \leq \#I$.
\end{Lemma} 
\begin{proof}
    We note that, for each $1\leq i\leq r$, the matrix 
\begin{equation}
    \Phi(i,\,j_i)^{(-R+l_i)}\Phi(i,\,j_i)^{(-R+2l_i)}\cdots \Phi(i,\,j_i)^{(-l_i)}\Phi(i,\,j_i)
\end{equation}
where
\begin{equation}    
\Phi(i,\,j_i):=((t-\theta)) \oplus \bigoplus_{j^\prime \leq j_i}\Phi_{l_i}[\mathbf{s}_{j^\prime}]
\end{equation} 
defines a pre-$t$-motive $M(i,\,j_i)$ and has a rigid analytic trivialization
\begin{equation}    
\Psi(i,\,j_i):=(\Omega_{l_i}) \oplus \bigoplus_{j^\prime \leq j_i}\Psi_{l_i}[\mathbf{s}_{j^\prime}].
\end{equation}
In the notation introduced in Section \ref{subsectionPapanikolas}, the algebraic group \begin{equation}
    \Gamma_{\Psi(1,\,j_1)\oplus \Psi(2,\,j_2)\oplus \cdots \oplus \Psi(r,\,j_r)}
\end{equation}
was characterized to be the smallest closed subscheme which contains the $\mathbb{L}\otimes_{\overline{K}(t)}\mathbb{L}$-valued point
\begin{equation}
    \widetilde{(\Psi(1,\,j_1)\oplus \Psi(2,\,j_2)\oplus \cdots \oplus \Psi(r,\,j_r))}=\widetilde{\Psi}(1,\,j_1)\oplus \widetilde{\Psi}(2,\,j_2)\oplus \cdots \oplus \widetilde{\Psi}(r,\,j_r),
\end{equation}
which can be easily shown to be an $\mathbb{L}\otimes_{\overline{K}(t)}\mathbb{L}$-valued point of 
\begin{equation}
    \Gamma_{\Psi(1,\,j_1)}\times \Gamma_{\Psi(2,\,j_2)} \times \cdots \times \Gamma_{\Psi(r,\,j_r)}
\end{equation}
as we have $\widetilde{\Psi}(i,\,j_i)\in \Gamma_{\Psi(i,\,j_i)}(\mathbb{L}\otimes_{\overline{K}(t)}\mathbb{L})$ for any $i$ by definitions. Therefore, we have an inclusion
\begin{equation}
    \Gamma_{\Psi(1,\,j_1)\oplus \Psi(2,\,j_2)\oplus \cdots \oplus \Psi(r,\,j_r)} \subset\Gamma_{\Psi(1,\,j_1)}\times \Gamma_{\Psi(2,\,j_2)} \times \cdots \times \Gamma_{\Psi(r,\,j_r)}.
\end{equation}
Since Theorem \ref{ThmPapanikolas} yields the isomorphisms
\begin{align}
    \Gamma_{\Psi(1,\,j_1)\oplus \Psi(2,\,j_2)\oplus \cdots \oplus \Psi(r,\,j_r)}&\simeq \Gamma_{M(1,\,j_1)\oplus M(2,\,j_2)\oplus \cdots \oplus M(r,\,j_r)} \text{ and }\\
    \Gamma_{\Psi(i,\,j_i)}&\simeq \Gamma_{M(i,\,j_i)}
\end{align}
for each $1 \leq i \leq r$, we obtain the closed immersion \eqref{mainimmersion}.
\end{proof}

For $0 \leq j_1,\,\dots,\, j_r \leq \# I$, we simply write
\begin{equation}\label{AbbGamma}
    \Gamma_{j_1,\,j_2,\,\dots,\,j_r}:=\Gamma_{M(1,\,j_1)\oplus M(2,\,j_2)\oplus \cdots \oplus M(r,\,j_r)}.
\end{equation}
\begin{Proposition}\label{CorTannakiangeneral}
    Let us take $0 \leq j_i^\prime \leq j_i \leq \#I$ for each $1\leq i \leq r$. Then, we have a faithfully flat homomorphism
    \begin{equation}\label{TannakianProjgeneral}
        \Gamma_{j_1,\,\dots,\,j_r} \twoheadrightarrow \Gamma_{j_1^\prime,\,\dots,\,j_r^\prime}
    \end{equation}
and commutative diagram
\begin{equation}\label{wholedirectproductdiagram}
\begin{tikzpicture}[auto]
\node (03) at (0, 2) {$\Gamma_{j_1,\,\dots,\,j_r}$}; \node (33) at (7, 2) {$\Gamma_{j_1^\prime,\,\dots,\,j_r^\prime}$}; 
\node (00) at (0, 0) {$\Gamma_{M(1,\,j_1)}\times\cdots \times \Gamma_{M(r,\,j_r)}$}; \node (30) at (7, 0) {$\Gamma_{M(1,\,j_1^\prime)}\times\cdots \times \Gamma_{M(r,\,j_r^\prime)}$}; 
\draw[->>] (03) to node {}(33);
\draw[->>] (00) to node {$\Pi$}(30);
\draw[{Hooks[right]}->] (03) to node {$\iota$} (00);
\draw[{Hooks[right]}->] (33) to node {$\iota$} (30);
\end{tikzpicture}
\end{equation}
where the homomorphism $\Pi$ is the direct product of homomorphisms $\Gamma_{M(i,\,j_i)}\twoheadrightarrow \Gamma_{M(i,\,j_i^\prime)}$ $(1 \leq i \leq r)$ given in  Lemma \ref{LemmaTannakianProjfixedl} and vertical lines are closed immersions given in Lemma \ref{Lemmainclusion}.
\end{Proposition}

\begin{proof}
    As the pre-$t$-motive $M(1,\,j_1^\prime)\oplus \cdots \oplus M(r,\,j_r^\prime)$ is a direct summand of $M(1,\,j_1)\oplus \cdots \oplus M(r,\,j_r)$, the Tannakian category $ \langle M(1,\,j_1^\prime)\oplus \cdots \oplus M(r,\,j_r^\prime) \rangle$ can be seen as a full Tannakian subcategory of $ \langle M(1,\,j_1)\oplus \cdots \oplus M(r,\,j_r) \rangle$. Therefore, we obtain faithfully flat homomorphism in \eqref{TannakianProjgeneral}. The commutativity of the diagram follows from Lemma \ref{LemmaTannakianProj}.
\end{proof}

The morphisms $\varphi$, $\psi$, and $\phi$ in Example \ref{Example1} can be seen as examples of faithfully flat morphisms given in Proposition \ref{CorTannakiangeneral}.
In proving that Immersion \eqref{mainimmersion} is an isomorphism, we study unipotent radicals of $t$-motivic Galois groups. Considering unipotent radicals, we deduce the following from Lemma \ref{lemmaexolicit} and Proposition \ref{CorTannakiangeneral}:
\begin{Corollary}\label{LemmaTannakianprojuniradical}
For any $0 \leq j_1,\,\dots,\,j_r \leq \#I$, we let $\mathcal{U}_{j_1,\,\dots,\,j_r}$ be the kernel of the morphism 
$
    \pi:\Gamma_{j_1,\,\dots,\,j_r}\twoheadrightarrow \Gamma_{0,\,\dots,\,0}
$
given by Proposition \ref{CorTannakiangeneral};
\begin{equation}\label{DefcalU}
    \mathcal{U}_{j_1,\,\dots,\,j_r}:=\ker \left( \pi:\Gamma_{j_1,\,\dots,\,j_r}\twoheadrightarrow \Gamma_{0,\,\dots,\,0}\right).
\end{equation}
\begin{enumerate}
    \item 
    
     The closed immersion \begin{equation}
    \iota:\Gamma_{j_1,\,\dots,\,j_r} \hookrightarrow\Gamma_{M(1,\,j_1)}\times \Gamma_{M(2,\,j_2)} \times \cdots \times \Gamma_{M(r,\,j_r)}
\end{equation}
given by Lemma \ref{Lemmainclusion} induces the embedding of algebraic group $\mathcal{U}_{j_1,\,\dots,\,j_r}$ into the direct product
    $U_{1,\,j_1}\times \cdots \times U_{r,\,j_r}$ (see Definition \ref{DefUandV} for  $U_{1,\,j_1},\, \dots ,$ $ U_{r,\,j_r}$).

\item If we take $0 \leq j_i^\prime \leq j_i \leq \#I$ for each $1\leq i \leq r$, then the homomorshism
    \begin{equation}\label{TannakianProjgeneral2}
        \Gamma_{j_1,\,\dots,\,j_r} \twoheadrightarrow \Gamma_{j_1^\prime,\,\dots,\,j_r^\prime}
    \end{equation} 
    given by Proposition \ref{CorTannakiangeneral} induces faithfully flat morphism
\begin{equation}
    \mathcal{U}_{j_1,\,\dots,\,j_r} \twoheadrightarrow \mathcal{U}_{j_1^\prime,\,\dots,\,j_r^\prime}.
\end{equation}

\item Let us take $0 \leq j_i^\prime \leq j_i \leq \#I$ for $1 \leq i \leq r$. Then the following diagram is commutative:
    \begin{equation}\label{wholedirectproductdiagramunipotent}
\begin{tikzpicture}[auto]
\node (03) at (0, 2) {$\mathcal{U}_{j_1,\,\dots,\,j_r}$}; \node (33) at (7, 2) {$\mathcal{U}_{j_1^\prime,\,\dots,\,j_r^\prime}$}; 
\node (00) at (0, 0) {$U_{1,\,j_1}\times\cdots \times U_{r,\,j_r}$}; \node (30) at (7, 0) {$U_{1,\,j_1^\prime}\times\cdots \times U_{r,\,j_r^\prime}$}; 
\draw[->>] (03) to node {}(33);
\draw[->>] (00) to node {$\Pi^\prime$}(30);
\draw[{Hooks[right]}->] (03) to node {$\iota$} (00);
\draw[{Hooks[right]}->] (33) to node {$\iota$} (30);
\end{tikzpicture}
\end{equation}
where the homomorphism $\Pi^\prime$ is the restriction of $\Pi$ in Diagram \eqref{wholedirectproductdiagram}.
    \end{enumerate} 
\end{Corollary}
\begin{proof}

Proposition \ref{CorTannakiangeneral} yields the immersion $\mathcal{U}_{j_1,\,\dots,\,j_r}\hookrightarrow U_{1,\,j_1}\times \cdots \times U_{r,\,j_r}$, see the following diagram of exact sequences:
\begin{small}
\begin{equation}
                \begin{tikzpicture}[auto] 
                   \node (22) at (-1, 2) {$\mathcal{U}$}; \node (42) at (3.5, 2) {$\Gamma_{j_1,\,\dots,\,j_r}$}; \node (62) at (8, 2) {$\Gamma_{0,\,\dots,\,0}$};
                    \node (20) at (-1, 0) {$U_{1,\,j_1}\times \cdots \times U_{r,\,j_r}$}; \node (40) at (3.5, 0) {$\Gamma_{M(1,\,j_1)}\times  \cdots \times \Gamma_{M(r,\,j_r)}$}; \node (60) at (8, 0) {$\Gamma_{M(1,\,0)}\times \cdots \times \Gamma_{M(r,\,0)}$};
    
                    \draw[{Hooks[right]}->] (62) to node {}(60);
                    \draw[{Hooks[right]}->] (42) to node {$\iota$}(40);
                    
                    \draw[{Hooks[right]}->] (22) to node {}(42);
                    \draw[{Hooks[right]}->] (20) to node {}(40);
    
                    \draw[->>] (42) to node {$\pi$}(62);
                    \draw[->>] (40) to node {$\Pi$}(60);
    
    
                \end{tikzpicture}.
            \end{equation}
            \end{small}
     The surjectivity of $\mathcal{U}_{j_1,\,\dots,\,j_r} \twoheadrightarrow \mathcal{U}_{j_1^\prime,\,\dots,\,j_r^\prime}$ follows from the diagram of exact sequences
    \begin{equation}
                \begin{tikzpicture}[auto] 
                   \node (22) at (-1, 2) {$\mathcal{U}_{j_1,\,\dots,\,j_r}$}; \node (42) at (3.5, 2) {$\Gamma_{j_1,\,\dots,\,j_r}$}; \node (62) at (8, 2) {$\Gamma_{0,\,\dots,\,0}$};
                    \node (20) at (-1, 0) {$\mathcal{U}_{j_1^\prime,\,\dots,\,j_r^\prime}$}; \node (40) at (3.5, 0) {$\Gamma_{j_1^\prime,\,\dots,\,j_r^\prime}$}; \node (60) at (8, 0) {$\Gamma_{0,\,\dots,\,0}$};
    
                    \draw[-] (62) to node[sloped] {$=$}(60);
                    \draw[->>] (42) to node {}(40);
                    \draw[->] (22) to node {}(20);
                    
                    \draw[{Hooks[right]}->] (22) to node {}(42);
                    \draw[{Hooks[right]}->] (20) to node {}(40);
    
                    \draw[->>] (42) to node {}(62);
                    \draw[->>] (40) to node {}(60);
    
    
                \end{tikzpicture}
            \end{equation}
            where the homomorphisms between $t$-motivic Galois groups are given by Proposition \ref{CorTannakiangeneral}.
 The third assertion follows from the commutativity of Diagram \eqref{wholedirectproductdiagram}.
\end{proof}

\subsection{Direct product decompositions of the $t$-motivic Galois group}\label{sscal}
Using the observations made in previous subsections, we prove that the inclusion in Lemma \ref{Lemmainclusion} is an equality (Theorem \ref{GaloisGroupDirectProd}).
We deduce Theorem \ref{main} from Theorem \ref{GaloisGroupDirectProd} at the end of this subsection.
Let us recall the notation in Equation \eqref{AbbGamma}.

\begin{Theorem}\label{GaloisGroupDirectProd}
Take a finite set $I$ of indices which satisfies Equation \eqref{subclosed} and fix its enumeration $I=\{ \mathbf{s}_1,\,\dots,\,\mathbf{s}_{\# I} \}$ such that $\operatorname{dep}\mathbf{s}_{j^\prime} \leq \operatorname{dep}\mathbf{s}_{j}$ for any $1 \leq j^\prime < j \leq \#I$. 
Let $l_1,\,\dots,\,l_r$ be distinct positive integers and $R$ be their common multiple.
Also take an element $u_{i,\,s}$ of $\overline{K}[t]$ which satisfies the inequality 
\begin{equation}\label{convergencecondition2}
    ||u_{i,\,s}||_\infty<|\theta|_\infty^{\frac{s p^{l_i}}{p^{l_i}-1}}
\end{equation}
for each
$1\leq i \leq r$ and $s \geq 1$.
For any choice of $0 \leq j_1,\,j_2,\,\dots,\,j_r \leq \#I$, the closed immersion 
$
    \Gamma_{j_1,\,j_2,\,\dots,\,j_r} \subset\Gamma_{M(1,\,j_1)}\times \Gamma_{M(2,\,j_2)} \times \cdots \times \Gamma_{M(r,\,j_r)}
$
given by Lemma \ref{Lemmainclusion} is an isomorphism and hence the following equation of algebraic groups holds:
\begin{equation}
    \Gamma_{j_1,\,j_2,\,\dots,\,j_r}=\Gamma_{M(1,\,j_1)}\times \Gamma_{M(2,\,j_2)} \times \cdots \times \Gamma_{M(r,\,j_r)}. \label{mainthmeq}
\end{equation}
\end{Theorem}


This subsection is devoted to the proof of Theorem \ref{GaloisGroupDirectProd}. Hence, in what follows, we fix a finite set $I$ of indices and its enumeration $I=\{ \mathbf{s}_1,\,\dots,\,\mathbf{s}_{\# I} \}$ satisfying the assumptions of the theorem. 
Distinct positive integers $l_1,\,\dots,\,l_r$, their common multiple $R$, and elements $u_{i,\,s}$ of $\overline{K}[t]$ which satisfies the Inequalities \eqref{convergencecondition2}
 are also taken and fixed for all
$1\leq i \leq r$ and $s \geq 1$.
We prove this theorem by the induction on the sum $ j_1+j_2+\cdots+j_r$.
We have already obtained Equation \eqref{mainthmeq} in the case where $j_1=\cdots=j_r=0$:
\begin{Proposition}\label{GaloisGroupDirectProd0}
    The following equality of algebraic groups holds:
    \begin{equation}
    \Gamma_{0,\,0,\,\dots,\,0}=\Gamma_{M(1,\,0)}\times \Gamma_{M(2,\,0)} \times \cdots \times \Gamma_{M(r,\,0)}.
    \label{mainthmeq0}
\end{equation}
\end{Proposition}
This can be seen as a paraphrase of Denis' theorem~(\cite{Denis1998}),
see Equation \eqref{CarlitzOnly}.

\subsubsection{The case of $\dim \Gamma_{M(i,\,j_i)}= \dim\Gamma_{M(i,\,j_i-1)} $ for some $ 1\leq i \leq r$.}

First, we deal with the simple case where $\dim \Gamma_{M(i,\,j_i)}= \dim\Gamma_{M(i,\,j_i-1)} $ for some $ 1\leq i \leq r$. Without loss of generality, we may assume that $i=1$.
The induction hypothesis implies the following equation: 
\begin{equation}
    \Gamma_{j_1-1,\,j_2,\,\dots,\,j_r}=\Gamma_{M(1,\,j_1-1)}\times \Gamma_{M(2,\,j_2)} \times \cdots \times \Gamma_{M(r,\,j_r)}.\label{inductionhyp}
\end{equation}

\begin{Proposition}\label{Proptrivialcase}
    Take $1 \leq j_1 \leq \#I$ and $0 \leq j_2,\,\dots,\,j_r \leq \#I$. Suppose that we have  $\dim \Gamma_{M(1,\,j_1)}= \dim\Gamma_{M(1,\,j_1-1)} $ and Equation \eqref{inductionhyp}.
Then we also obtain
\begin{equation} 
    \Gamma_{j_1,\,\dots,\,j_r}=\Gamma_{M(1,\,j_1)}\times \Gamma_{M(2,\,j_2)} \times \cdots \times \Gamma_{M(r,\,j_r)}.
\end{equation}
    
\end{Proposition}
\begin{proof}

Because of Equation \eqref{inductionhyp} and the closed immersion \eqref{mainimmersion}, we have
\begin{align}
\dim \Gamma_{j_1,\,j_2,\,\dots,\,j_r} & \leq \dim \left(\Gamma_{M(1,\,j_1)}\times \Gamma_{M(2,\,j_2)} \times \cdots \times \Gamma_{M(r,\,j_r)}\right)\\
&=\dim\left(\Gamma_{M(1,\,j_1-1)}\times \Gamma_{M(2,\,j_2)} \times \cdots \times \Gamma_{M(r,\,j_r)} \right)\\
&=\dim \Gamma_{j_1-1,\,j_2,\,\dots,\,j_r}.
\end{align}
By Proposition \ref{CorTannakiangeneral}, we have faithfully flat morphism
$
\overline{\psi}:\Gamma_{j_1,\,j_2,\,\dots,\,j_r}$ $\twoheadrightarrow \Gamma_{j_1-1,\,j_2,\,\dots,\,j_r}\label{eqmorphosi}
$
of algebraic groups, which yields the inequality 
$
    \dim \Gamma_{j_1,\,j_2,\,\dots,\,j_r} \geq \dim \Gamma_{j_1-1,\,j_2,\,\dots,\,j_r},
$
hence we have
$
\dim \Gamma_{j_1,\,\dots,\,j_r} = \dim \big(\Gamma_{M(1,\,j_1)} \times  \Gamma_{M(2,\,j_2)} \times \cdots$  $\times \Gamma_{M(r,\,j_r)}\big).
$
As algebraic groups $\Gamma_{M(1,\,j_1)},\, \dots,\,\Gamma_{M(r-1,\,j_{r-1})},$ and $\Gamma_{M(r,\,j_r)}$ are connected and smooth over $\overline{\mathbb{F}_q(t)}$ (see Theorem \ref{ThmPapanikolas}), so is the direct product $\Gamma_{M(1,\,j_1)}\times \Gamma_{M(2,\,j_2)} \times \cdots \times \Gamma_{M(r,\,j_r)}$, and hence we can conclude that the closed immersion \eqref{mainimmersion} is an isomorphism.
\end{proof}

\subsubsection{The case where $j_i \geq 1$ for some $1 \leq i \leq r$
 and $j_1=\cdots=j_{i-1}=j_{i+1}=\cdots=j_r=0$}

Second, we consider the case where we have $1 \leq i \leq r$ such that $j_i \geq 1$ and $j_1=\cdots=j_{i-1}=j_{i+1}=\cdots=j_r=0$ with $\dim \Gamma_{M(i,\,j_i)}>\Gamma_{M(i,\,j_i-1)}$. 
Then we further have 
 \begin{equation}
    \dim \Gamma_{M(i,\,j_i)}= \dim\Gamma_{M(i,\,j_i-1)} +1,\label{assumptiondimgamma}
\end{equation}
see Formula \eqref{successor}.
We may assume $i=1$ without loss of generality.

For $0 \leq j \leq \#I$, Corollary \ref{CorTannakiangeneral} yields a faithfully flat morphism
\begin{equation}
\pi_j:\Gamma_{j,\,0,\,\dots,\,0}\twoheadrightarrow \Gamma_{0,\,\dots,\,0}\simeq \mathbb{G}_m^r
\end{equation}
(see also Equation \eqref{AbbGamma}) and we simply write $\mathcal{U}_j$ for the kernel of $\pi_j$, which we write as $\mathcal{U}_{j,\,0,\,\dots,\,0}$ in Corollary \ref{LemmaTannakianprojuniradical}. 
As $M(1,\,j)$ is a direct summand of $M(1,\,j)\oplus M(2,\,0)\oplus\cdots\oplus M(r,\,0)$,  Tannakian duality also yields the faithfully flat morphism
\begin{equation}
    \overline{\varphi}_j:\Gamma_{j,\,0,\,\dots,\,0}\twoheadrightarrow \Gamma_{M(1,\,j)},
\end{equation}
which fits into the following commutative diagram (Lemma \ref{LemmaTannakianProj}):
\begin{equation}
\begin{tikzpicture}[auto]
\node (03) at (0, 2) {$\Gamma_{j,\,0,\,\dots,\,0}$}; \node (33) at (5, 2) {$\Gamma_{M(1,\,j)}\times\Gamma_{M(2,\,0)}\times \cdots \times \Gamma_{M(r,\,0)}$}; 
\node (30) at (5, 0) {$\Gamma_{M(1,\,j_1)}$.}; 
\draw[{Hooks[right]}->] (03) to node {$\iota$}(33);
\draw[->>] (03) to node {} (30);
\draw[-] (30) to node {$\overline{\varphi}_{j_1}$} (03);
\draw[->>] (33) to node {$\operatorname{pr}_1$} (30);
\end{tikzpicture}
\end{equation}
Here, the immersion $\iota$ is that given by Lemma \ref{Lemmainclusion}. 
The restriction $\varphi_j:=\overline{\varphi}_j|_{\mathcal{U}_j}$ has the image contained in $U_{1,\,j}$ (Definition \ref{DefUandV}) by the commutative diagram
\begin{equation}\label{diagramcase2}
                \begin{tikzpicture}[auto]

    \node (1-3) at (-1, 2) {$1$}; \node (22) at (1, 2) {$\mathcal{U}_j$}; \node (42) at (4, 2) {$\Gamma_{j,\,0,\,\dots,\,0}$}; \node (62) at (7.5, 2) {$\mathbb{G}_m^r$};\node (82) at (9.5, 2) {$1$};
    \node (1-5) at (-1, 0) {$1$}; \node (20) at (1, 0) {$U_{1,\,j}$}; \node (40) at (4, 0) {$\Gamma_{M(1,\,j)} $}; \node (60) at (7.5, 0) {$\mathbb{G}_m$};\node (80) at (9.5, 0) {$1$};
    \draw[->>] (62) to node {$\operatorname{pr}_1 $}(60);
    
    \draw[->>] (42) to node {$\overline{\varphi}_j$}(40);
    \draw[->] (22) to node {$ \varphi_{j}$}(20);

    \draw[{Hooks[right]}->] (22) to node {}(42);
    \draw[{Hooks[right]}->] (20) to node {}(40);

    \draw[->>] (42) to node {$\pi_j$}(62);
    \draw[->>] (40) to node {}(60);

    \draw[->] (1-3) to node {}(22);
    \draw[->] (1-5) to node {}(20);
    
    \draw[->] (62) to node {}(82);
    \draw[->] (60) to node {}(80);
    \end{tikzpicture}
            \end{equation}
which holds for each $0 \leq j \leq \#I$. Here, the surjection $\Gamma_{M(1,\,j)}\twoheadrightarrow \mathbb{G}_m\simeq \Gamma_C$ is the homomorphism given in Lemma \ref{LemmaTannakianProjfixedl}. 
We note that $\varphi_j$ is faithfully flat if $\Gamma_{j,\,0,\,\dots,\,0}=\Gamma_{M(1,\,,j)}\times \Gamma_{M(2,\,0)}\times \cdots \times \Gamma_{M(r,\,0)}$ holds.
\begin{Lemma}\label{Lemmacase2surj}
    Take $j_1 \geq 1$ and assume Equation \eqref{assumptiondimgamma} holds with $i=1$. Then the equality
    \begin{equation}\label{inductionhypcase2}
    \Gamma_{j_1-1,\,0,\,\dots,\,0} =\Gamma_{M(1,\,j_1-1)}\times \Gamma_{M(2,\,0)} \times \cdots \times \Gamma_{M(r,\,0)}
\end{equation}
implies that the restriction $\varphi_{j_1}=\overline{\varphi}_{j_1}|_{\mathcal{U}_{j_1}}$ in Diagram \eqref{diagramcase2} is surjective onto $U_{1,\,j_1}$.
\end{Lemma}
\begin{proof}
    We consider the algebraic subgroup $V_{1,\,j_1}$ of $U_{1,\,j_1}$, see Definition \ref{DefUandV}. As we assumed that Equation \eqref{assumptiondimgamma} holds, we have 
    \begin{equation}
        \dim V_{1,\,j_1}=\dim U_{1,\,j_1}-\dim U_{1,\,j_1-1}=\dim \Gamma_{M(1,\,j_1)}-\dim \Gamma_{M(1,\,j_1-1)}=1
    \end{equation}
    and hence we have $ V_{1,\,j_1} \simeq \mathbb{G}_a$ by Lemma \ref{VinGa}. Let us take arbitrary $a \in \overline{\mathbb{F}_q(t)}$ and consider the corresponding element
    \begin{equation}
        v_{a}:=
    (1)\oplus I_{\operatorname{dep}(\mathbf{s}_1)+1} \oplus \cdots \oplus I_{\operatorname{dep}(\mathbf{s}_{j_1-1})+1} \oplus 
    \begin{pmatrix}
        1&0 & \cdots&\cdots& 0 \\
        0& 1&\ddots&&\vdots\\
        \vdots&0&\ddots&\ddots&\vdots\\
        0&\vdots&\ddots&\ddots&0\\
        a& 0&\cdots&0&1
    \end{pmatrix}
    \end{equation}
    of $V_{1,\,j_1}(\overline{\mathbb{F}_q(t)})$. By surjectivity of the homomorphism $\overline{\varphi}_{j_1}$, we can pick $$w_1 \in \Gamma_{j_1,\,0,\,\dots,\,0}(\overline{\mathbb{F}_q(t)})$$ such that $\overline{\varphi}_{j_1}(w_1)=v_a$. 
    We further take an arbitrary $b \in \overline{\mathbb{F}_q(t)}^\times =\mathbb{G}_m(\overline{\mathbb{F}_q(t)})$ and 
    $$w_2 \in \Gamma_{j_1,\,0,\,\dots,\,0}(\overline{\mathbb{F}_q(t)})$$
    which is mapped to $b$ by the composition $\operatorname{pr}_1 \circ \pi_{j_1}$, see Diagram \eqref{diagramcase2}. 
    Then the commutativity of the group $\mathbb{G}_m^r$ shows that the commutator $w_1 w_2 w_1^{-1} w_2^{-1}$ goes to the identity element via $\pi_{j_1}$ and hence we can conclude that $w_1 w_2 w_1^{-1} w_2^{-1}$ is an element of the kernel $\mathcal{U}_{j_1}$ of $\pi_{j_1}$. We also notice that the image $\overline{\varphi}_{j_1}(w_1 w_2 w_1^{-1} w_2^{-1})$ of the commutator is $v_{a(1-b^{-\operatorname{wt}(\mathbf{s}_{j_1})})}$ given by
    \begin{equation}
        (1)\oplus I_{\operatorname{dep}(\mathbf{s}_1)+1} \oplus \cdots \oplus I_{\operatorname{dep}(\mathbf{s}_{j_1-1})+1} \oplus 
    \begin{pmatrix}
        1&0 & \cdots&\cdots& 0 \\
        0& 1&&&\vdots\\
        \vdots&0&\ddots&&\vdots\\
        0&\vdots&&\ddots&0\\
        a(1-b^{-\operatorname{wt}(\mathbf{s}_{j_1})})& 0&\cdots&0&1
    \end{pmatrix},
    \end{equation}
     the element of $V_{1,\,j_1}(\overline{\mathbb{F}_q(t)})$ corresponding to $a(1-b^{-\operatorname{wt}(\mathbf{s}_{j_1})}) \in \overline{\mathbb{F}_q(t)}$ via the idenfication $V_{1,\,j_1} \simeq \mathbb{G}_a$. 
     As elements $a$ and $b$ are arbitrarily chosen, we can conclude that the image $\varphi_{j_1}(\mathcal{U}_{j_1})$ contains $V_{1,\,j_1}$.

    On the other hand, we have the commutative diagram
    \begin{equation}
\begin{tikzpicture}[auto]
\node (03) at (0, 2) {$\mathcal{U}_{j_1}$}; \node (33) at (3, 2) {$\mathcal{U}_{j_1-1}$}; 
\node (00) at (0, 0) {$U_{1,\,j_1}$}; \node (30) at (3, 0) {$U_{1,\,j_1-1}$,}; 
\draw[->>] (03) to node {}(33);
\draw[->>] (00) to node {}(30);
\draw[->] (03) to node {$\varphi_{j_1}$} (00);
\draw[->] (33) to node {$\varphi_{j_1-1}$} (30);
\end{tikzpicture}
\end{equation}
    and we can prove the surjectivity of $\varphi_{j_1-1}$ as we have Equation \eqref{inductionhypcase2}. Hence the subgroup $\varphi_{j_1}(\mathcal{U}_{j_1})$ of $U_{1,\,j_1}$ is mapped onto $U_{1,\,j_1}$. As $\varphi_{j_1}(\mathcal{U}_{j_1})$ contains the kernel $V_{1,\,j_1}$ of the surjection $U_{1,\,j_1} \twoheadrightarrow U_{1,\,j_1-1} $ we can conclude that $\varphi_{j_1}(\mathcal{U}_{j_1})=U_{1,\,j_1}$ by the correspondence theorem.
\end{proof}

Now we are ready to prove the following desired result:
\begin{Proposition}\label{GaloisGroupDirectProd2}
     If we have Equations \eqref{assumptiondimgamma} with $i=1$ and \eqref{inductionhypcase2} for some $j_1 \geq 1$, then
     \begin{equation}
    \Gamma_{j_1,\,0,\,\dots,\,0} =\Gamma_{M(1,\,j_1)}\times \Gamma_{M(2,\,0)} \times \cdots \times \Gamma_{M(r,\,0)}.
\end{equation}
\end{Proposition}
\begin{proof}
The upper exact sequence in Diagram \eqref{diagramcase2} shows that
\begin{equation}
    \dim \Gamma_{j_1,\,0,\,\dots,\,0}=\dim \mathcal{U}_{j_1}+\dim \mathbb{G}_m^r=\dim \mathcal{U}_{j_1}+r.
\end{equation}
As we assumed that Equations \eqref{assumptiondimgamma} and \eqref{inductionhypcase2} hold, we can use Lemma \ref{Lemmacase2surj} to obtain $\dim \mathcal{U}_{j_1} \geq \dim U_{1,\,j_1}=\dim \Gamma_{M(1,\,j_1)}-1$, hence we have
\begin{align}
    \dim \Gamma_{j_1,\,0,\,\dots,\,0}&\geq \dim \Gamma_{M(1,\,j_1)}-1+r\\
    &=\dim (\Gamma_{M(1,\,j_1)}\times \Gamma_{M(2,\,0)}\times \cdots \times \Gamma_{M(r,\,0)}).
\end{align}
Consequently we can see that the closed immersion $\Gamma_{j_1,\,0,\,\dots,\,0} \hookrightarrow \Gamma_{M(1,\,j_1)}\times \Gamma_{M(2,\,0)}\times \cdots \times \Gamma_{M(r,\,0)}$ in Lemma \ref{Lemmainclusion} is an isomorphism as the latter group is a direct product of smooth and connected algebraic groups (Theorem \ref{ThmPapanikolas}).
\end{proof}

\subsubsection{The case where $j_i,\,j_{i^\prime} \geq 1$ for some $1\leq i < i^\prime \leq r$}

Third, we verify Equation \eqref{mainthmeq} in the the case where $j_i,\,j_{i^\prime} \geq 1$ for some $1\leq i < i^\prime \leq r$. Ideas in the arguments here are similar to those in Example \ref{Example1}.
We can assume that $i=1$ and $i^\prime =2$ without loss of generality.

Let us recall the notation in Equations \eqref{AbbGamma} and \eqref{DefcalU}.
By the induction hypothesis, we may assume
    \begin{equation} 
        \Gamma_{ j_1^\prime,\,j_2^\prime,\,j_3\dots,\,j_r}
        =\Gamma_{M(1,\,j_1^\prime)}\times \Gamma_{M(2,\,j_2^\prime)} \times  \Gamma_{M(3,\,j_3)} \times \cdots \times \Gamma_{M(r,\,j_r)}
    \end{equation}    
and hence we have
\begin{equation}\label{inductionhypcase3unipotentpart}
    \mathcal{U}_{j_1^\prime,\,j_2^\prime,\,j_3,\,\dots,\,j_r}=U_{1,\,j_1^\prime}\times U_{2,\,j_2^\prime}\times U_{3,\,j_3}\times \cdots \times U_{r,\,j_r}
\end{equation}
(see Definition \ref{DefUandV} and Corollary \ref{LemmaTannakianprojuniradical}) for $(j_1^\prime,\,j_2^\prime)=(j_1-1,\,j_2),\,(j_1,\,j_2-1)$, and $(j_1-1,\,j_2-1)$.

Considering the direct product of the surjections given in Lemma \ref{LemmaTannakianProjfixedl}, we have a surjection
\begin{small}
\begin{align}
    \Gamma_{M(1,\,j_1)}\times  \cdots \times \Gamma_{M(r,\,j_r)}
    &\twoheadrightarrow
    \Gamma_{M(1,\,j_1-1)}\times \Gamma_{M(2,\,j_2-1)}\times \Gamma_{M(3,\,j_3)}\times \cdots \times\Gamma_{M(r,\,j_r)}
\end{align}
\end{small}
which yields the surjective homomorphism
\begin{equation}
    U_{1,\,j_1}\times \cdots \times U_{r,\,j_r}\twoheadrightarrow U_{1,\,j_1-1}\times U_{2,\,j_2-1}\times U_{3,\,j_3}\cdots \times U_{r,\,j_r},
\end{equation}
whose kernel is $V_{1,\,j_1}\times V_{2,\,j_2}$, see Definition \ref{DefUandV}. 
If we let $\mathcal{V}$ be the kernel of the surjection $\mathcal{U} \twoheadrightarrow \mathcal{U}_{j_1-1,\,j_2-1,\,j_3,\,\dots,\,j_r}$ given by Corollary \ref{LemmaTannakianprojuniradical}, then $\mathcal{V}$ can be immersed to $V_{1,\,j_1}\times V_{2,\,j_2}$ via the immersion $\mathcal{U} \hookrightarrow U_{1,\,j_1} \times \cdots \times U_{r,\,j_r}$ in Corollary \ref{LemmaTannakianprojuniradical}, see the following diagram: 
\begin{small}
\begin{equation}    \label{upperrectangle}
\begin{tikzpicture}[auto] 
                   \node (22) at (-0.4, 2) {$\mathcal{V}$}; \node (42) at (3, 2) {$\mathcal{U}$}; \node (62) at (8, 2) {$\mathcal{U}_{j_1-1,\,j_2-1,\,j_3,\,\dots,\,j_r}$};
                    \node (20) at (-0.4, 0) {$V_{1,\,j_1}\times V_{2,\,j_2}$}; \node (40) at (3, 0) {$U_{1,\,j_1}\times \cdots \times U_{r,\,j_r}$}; \node (60) at (8, 0) {$U_{1,\,j_1-1}\times U_{2,\,j_2-1}\times U_{3,\,j_3}\times \cdots \times U_{r,\,j_r}$};
    
                    \draw[-] (62) to node[sloped] {$=$}(60);
                    \draw[{Hooks[right]}->] (42) to node {}(40);
                    \draw[{Hooks[right]}->] (22) to node {}(20);
                    
                    \draw[{Hooks[right]}->] (22) to node {}(42);
                    \draw[{Hooks[right]}->] (20) to node {}(40);
    
                    \draw[->>] (42) to node {}(62);
                    \draw[->>] (40) to node {}(60);
    
    
                \end{tikzpicture}.
            \end{equation}
\end{small}

\begin{Lemma}\label{MappedOnto}
    Assume that Equation \eqref{inductionhypcase3unipotentpart} holds for $(j_1^\prime,\,j_2^\prime)=(j_1-1,\,j_2),\,(j_1,\,j_2-1)$, and $(j_1-1,\,j_2-1)$. Then the restriction of the projection $\operatorname{pr}_i:V_{1,\,j_1}\times V_{2,\,j_2} \twoheadrightarrow V_{i,\,j_i}$ to the algebraic subgroup $\mathcal{V}$ is surjective for $i=1,\,2$.
\end{Lemma}
\begin{proof}
    It is enough to consider the case $i=1$. We have the following commutative diagram whose horizontal lines are exact:
    
\begin{equation}\label{Diagram2217}
                \begin{tiny}
                \begin{tikzpicture}[auto] 
                   \node (22) at (-0.7, 1.5) {$\mathcal{V}$}; \node (42) at (3, 1.5) {$\mathcal{U}$}; \node (62) at (8, 1.5) {$\mathcal{U}_{j_1-1,j_2-1,\,j_3,\,\dots,\,j_r}$};
                    \node (20) at (-0.7, 0) {$V_{1,\,j_1}\times V_{2,\,j_2}$}; \node (40) at (3, 0) {$U_{1,j_1}\times \cdots \times U_{r,\,j_r}$}; \node (60) at (8, 0) {$U_{1,\,j_1-1}\times U_{2,\,j_2-1} \times U_{3,\,j_3}\times \cdots \times U_{r,\,j_r}$};
                    \node (2-2) at (-0.7, -1.5) {$V_{1,\,j_1}$}; \node (4-2) at (3, -1.5) {$U_{1,\,j_1-1}\times U_{2,\,j_2}\times \cdots \times U_{r,\,j_r}$}; \node (6-2) at (8, -1.5) {$U_{1,\,j_1-1}\times U_{2,\,j_2-1} \times U_{3,\,j_3}\times \cdots \times U_{r,\,j_r}$};
                    \node (2-4) at (-0.7, -3) {$V_{1,\,j_1}$}; \node (4-4) at (3, -3) {$\mathcal{U}_{j_1-1,j_2,\,\dots,\,j_r}$}; \node (6-4) at (8, -3) {$\mathcal{U}_{j_1-1,j_2-1,\,j_3,\,\dots,\,j_r}$};
    
                    \draw[=] (62) to node[sloped] {$= $}(60);
                    \draw[=] (60) to node[sloped] {$= $}(6-2);
                    
                    \draw[->>] (40) to node {}(4-2);
                    \draw[{Hooks[right]}->] (42) to node {$\iota$}(40);
                    \draw[{Hooks[right]}->] (22) to node {$\iota \mid_{\mathcal{V}}$}(20);
                    
                    \draw[{Hooks[right]}->] (22) to node {}(42);
                    \draw[{Hooks[right]}->] (20) to node {}(40);
                    \draw[{Hooks[right]}->] (2-2) to node {}(4-2);
    
                    \draw[->>] (42) to node {$\pi$}(62);
                    \draw[->>] (4-2) to node {}(6-2);
                    \draw[->>] (40) to node {$\Pi$}(60);

                    \draw[->>] (20) to node {$\operatorname{pr}_1$}(2-2);
                    \draw[=] (2-2) to node[sloped] {$= $}(2-4);
                    \draw[=] (4-2) to node[sloped] {$= $}(4-4);
                    \draw[=] (6-2) to node[sloped] {$= $}(6-4);
                    \draw[->>] (4-4) to node {}(6-4);
                    \draw[{Hooks[right]}->] (2-4) to node {}(4-4);
    
                \end{tikzpicture}.
            \end{tiny}
            \end{equation}
            Lemma \ref{LemmaTannakianProj} shows that the composed homomorphism from $\mathcal{U}_{j_1,\,\dots,\,j_r}$, which is simply written as $\mathcal{U}$ in the Diagram above \eqref{Diagram2217}, to $\mathcal{U}_{j_1-1,\,j_2,\,\dots,\,j_r}$ in the middle column of the diagram coincides with the surjection given in Corollary \ref{LemmaTannakianprojuniradical} (2). Hence $\mathcal{U}$ is mapped onto $\mathcal{U}_{j_1-1,\,j_2,\dots,\,j_r}$. The commutativity of the diagram proves that the composition $(\operatorname{pr}_1 \circ (\iota|_{\mathcal{V}})): \mathcal{V} \rightarrow V_{1,\,j_1}$ is also surjective as the horizontal lines are exact. So $\mathcal{V}$ is mapped onto $V_{1,\,j_1}$ by the projection $\operatorname{pr}_1:V_{1,\,j_1}\times V_{2,\,j_2} \twoheadrightarrow V_{1,\,j_1}$.
\end{proof}

In order to verify Equation \eqref{mainthmeq}, it is suffices to consider the case where we have $\dim \Gamma_{M(i,\,j_i)}> \dim\Gamma_{M(i,\,j_i-1)} +1$ for $i=1,\,2$ because of Proposition \ref{Proptrivialcase}. These inequalities imply
 \begin{equation}
    \dim \Gamma_{M(i,\,j_i)}= \dim\Gamma_{M(i,\,j_i-1)} +1 \quad (i=1,\,2),\label{assumptiondimgammacase3}
\end{equation}
see Formula \eqref{successor}, and that the algebraic groups $V_{1,\,j_1}$ and $V_{2,\,j_2}$ are isomorphic to $\mathbb{G}_a$.
Indeed, we have
\begin{equation}
    \dim V_{i,\,j_i}= \dim U_{i,\,j_i}-\dim U_{i,\,j_i-1}= \dim \Gamma_{M(i,\,j_i)}-\dim \Gamma_{M(i,\,j_i-1)}=1
\end{equation}
and hence $V_{i,\,j_i}\simeq \mathbb{G}_a$ by Corollary \ref{VinGa} for $i=1,\,2$.
\begin{Lemma}\label{Lemmaaction}
    Assume that Equality \eqref{assumptiondimgammacase3} holds and consider the $\mathbb{G}_m^2(\overline{\mathbb{F}_q(t)})$-action on $(V_{1,\,j_1}\times V_{2,\,j_2})(\overline{\mathbb{F}_q(t)})=\mathbb{G}_a^2(\overline{\mathbb{F}_q(t)})$ given by
    \begin{equation}
        (a_1,\,a_2).(x_1,\,x_2):=(a_1^{\operatorname{wt}(\mathbf{s}_{j_1})}x_1,\,a_2^{\operatorname{wt}(\mathbf{s}_{j_2})}x_2)
    \end{equation}
    for given $(a_1,\,a_2) \in \mathbb{G}_m^2(\overline{\mathbb{F}_q(t)})=(\overline{\mathbb{F}_q(t)}^\times)^2$ and $(x_1,\,x_2) \in (V_{1,\,j_1}\times V_{2,\,j_2})(\overline{\mathbb{F}_q(t)})=\overline{\mathbb{F}_q(t)}^2$.
    Then the subgroup $\mathcal{V}(\overline{\mathbb{F}_q(t)})$ is closed under this action.
\end{Lemma}
\begin{proof}
    Let us consider the diagram
    \begin{small}\begin{equation}\label{diagramcase3action}
                \begin{tikzpicture}[auto]

    \node (22) at (-1, 2) {$\mathcal{U}$}; \node (42) at (4, 2) {$\Gamma_{ j_1,\,j_2,\,\dots,\,j_r}$}; \node (62) at (9, 2) {$\Gamma_{ 0,\,0,\,\dots,\,0}$};
     \node (20) at (-1, 0) {$\mathcal{U}_{j_1-1,\,j_2-1,\,j_3,\,\dots,\,j_r}$}; \node (40) at (4, 0) {$\Gamma_{ j_1-1,\,j_2-1,\,j_3,\,\dots,\,j_r}$}; \node (60) at (9, 0) {$\Gamma_{ 0,\,0,\,\dots,\,0}$};
    \draw[=] (62) to node [sloped]{$=$}(60);
    
    \draw[->>] (42) to node {}(40);
    \draw[->] (22) to node {}(20);

    \draw[{Hooks[right]}->] (22) to node {}(42);
    \draw[{Hooks[right]}->] (20) to node {}(40);

    \draw[->>] (42) to node {}(62);
    \draw[->>] (40) to node {}(60);

    \end{tikzpicture},
            \end{equation}
            \end{small}
by which we can consider the conjugate action of $\Gamma_{ j_1,\,j_2,\,\dots,\,j_r}(\overline{\mathbb{F}_q(t)})$ on $\mathcal{U}(\overline{\mathbb{F}_q(t)})$. 
Written explicitly, the action is given by $R.Q:=R^{-1}QR$ for arbitrary matrices $R \in \Gamma_{ j_1,\,j_2,\,\dots,\,j_r}(\overline{\mathbb{F}_q(t)})$ and $Q \in \mathcal{U}(\overline{\mathbb{F}_q(t)})$.
We can also consider the $\Gamma_{ j_1,\,j_2,\,\dots,\,j_r}(\overline{\mathbb{F}_q(t)})$-action on the group $\mathcal{U}_{j_1-1,\,j_2-1,\,j_3,\,\dots,\,j_r}(\overline{\mathbb{F}_q(t)})$ by the conjugate action of the group $\Gamma_{ j_1,\,j_2-1,\,j_3,\,\dots,\,j_r}(\overline{\mathbb{F}_q(t)})$ and the surjection $ \Gamma_{ j_1,\,j_2,\,\dots,\,j_r} \twoheadrightarrow \Gamma_{ j_1-1,\,j_2-1,\,j_3,\,\dots,\,j_r}$, which is given by  Proposition \ref{CorTannakiangeneral}. As the homomorphism $\mathcal{U}\twoheadrightarrow \mathcal{U}_{j_1-1,\,j_2-1,\,j_3,\,\dots,\,j_r}$ in the above diagram is $\Gamma_{ j_1,\,j_2,\,\dots,\,j_r}(\overline{\mathbb{F}_q(t)})$-equivariant, its kernel $\mathcal{V}(\overline{\mathbb{F}_q(t)})$ is closed under the action. We note that all elements of $V_j(\overline{\mathbb{F}_q(t)})$ commute with any element of $U_j(\overline{\mathbb{F}_q(t)})$ (see Subsection \ref{eachlevel} for the definitions), thus $\mathcal{V}(\overline{\mathbb{F}_q(t)})$ is contained in the center of $\mathcal{U}(\overline{\mathbb{F}_q(t)})$ and hence is equipped with the $\Gamma_{ 0,\,0,\,\dots,\,0}(\overline{\mathbb{F}_q(t)})$-action induced by the action of $\Gamma_{ j_1,\,j_2,\,\dots,\,j_r}(\overline{\mathbb{F}_q(t)})$. 
We can check, by the manner similar to Equation \eqref{Examplecommutator}, that this $\Gamma_{ 0,\,0,\,\dots,\,0}(\overline{\mathbb{F}_q(t)})=\mathbb{G}_m^r(\overline{\mathbb{F}_q(t)})$-action is given by $(a_1,\,a_2,\,\dots,\,a_r).(x_1,\,x_2):=(a_1^{\operatorname{wt}(\mathbf{s}_{j_1})},\,a_2^{\operatorname{wt}(\mathbf{s}_{j_2})})$ for each $(a_1,\,a_2,\,\dots,\,a_r) \in \Gamma_{ 0,\,0,\,\dots,\,0}(\overline{\mathbb{F}_q(t)})$ and $(x_1,\,x_2)\in \mathcal{V}(\overline{\mathbb{F}_q(t)})$.
\end{proof}

By Lemmas \ref{MappedOnto} and \ref{Lemmaaction}, we can use the following lemma which shows that $\mathcal{V}$ equals $V_{1,\,j_1}\times V_{2,\,j_2}$:

\begin{Lemma}\label{LemmaWisWhole}
    Let $V_1$ and $V_2$ be algebraic groups over $\overline{\mathbb{F}_q(t)}$ isomorphic to $\mathbb{G}_a$ and consider the action of $\mathbb{G}_m^2(\overline{\mathbb{F}_q(t)})$ on the direct product $(V_1\times V_2)(\overline{\mathbb{F}_q(t)})$ given by
    \begin{equation}
        (\alpha,\,\beta).(x,\,y)=(\alpha^{w_1}x,\,\beta^{ w_2}y)
    \end{equation}
    with some $w_1,\,w_2 \geq 1$ for each $(\alpha,\,\beta) \in \mathbb{G}_m^2(\overline{\mathbb{F}_q(t)})$ and $(x,\,y)\in (V_1\times V_2)(\overline{\mathbb{F}_q(t)})$.
    Take an algebraic subgroup $W$ of $V_1 \times V_2$ which is mapped onto $V_1$ and onto $V_2$ via each projection.
    If the subgroup $W(\overline{\mathbb{F}_q(t)})$ is closed under the $\mathbb{G}_m^2(\overline{\mathbb{F}_q(t)})$-action on $(V_1 \times V_2)(\overline{\mathbb{F}_q(t)})$, then the algebraic subgroup $W$ must be equal to $V_1 \times V_2$.
\end{Lemma}
\begin{proof}
    As we have a surjection $W\twoheadrightarrow V_1 \simeq  \mathbb{G}_a$ by the assumption, the codimension of $W$ in $V_1 \times V_2$ is $0$ or $1$. Hence the algebraic set $W(\overline{\mathbb{F}_q(t)})$ is defined by a polynomial 
    \begin{equation}
        (b_m X^{p^m}+b_{m-1}X^{p^{m-1}}+\cdots +b_0X^{p^0})+(c_{n}Y^{ p^{n}}+c_{n-1}Y^{ p^{n-1}}+\cdots +c_0 Y^{ p^0})
    \end{equation}
    in variables $X$ and $Y$ with some $b_0,\,\cdots,\,b_m,\,c_0,\,\cdots,\,c_n \in \overline{\mathbb{F}_q(t)}$ , see \cite[Corollary 1.8]{Conrad2015}. 
    There exists $y\in \overline{\mathbb{F}_q(t)}$ such that $(1,\,y) \in W(\overline{\mathbb{F}_q(t)})$
    as we have assumed that $W$ is mapped onto $V_1$ via the first projection. 
    Take any $\alpha \in \overline{\mathbb{F}_q(t)}^\times$ and consider the action of $(\alpha,\,1)\in \mathbb{G}_m^2(\overline{\mathbb{F}_q(t)})$. As $W(\overline{\mathbb{F}_q(t)})$ is closed under the action, we have $(\alpha,\,1).(1,\,y) \in W(\overline{\mathbb{F}_q(t)})$ and hence it follows that
    \begin{align}
        0&=\left(b_m (\alpha^{w_1})^{p^m}+b_{m-1}(\alpha^{w_1})^{p^{m-1}}+\cdots +b_0(\alpha^{w_1})^{p^0}\right)\\
        &\quad \ \ +\left(c_{n}y^{p^{n}}+c_{n-1}y^{ p^{n-1}}+\cdots +c_0 y^{p^0}\right).
    \end{align}
    As $\alpha$ is an arbitrarily chosen non-zero element in the infinite field $\overline{\mathbb{F}_q(t)}$, it must follows that $b_0=b_1=\cdots=b_m=0$. By a similar argument, we can also obtain $c_0=c_1=\cdots=c_{n}=0$, hence the desired equality $W=V_1\times V_2$.
\end{proof}

    The conclusion of this subsubsection is the following:

\begin{Proposition}\label{GaloisGroupDirectProd3}
    Take $1 \leq j_1,\,j_2 \leq \#I$ and $0 \leq j_3,\,\dots,\,j_r \leq \#I$ and assume that the inequality
    \begin{equation}
        \dim \Gamma_{M(i,\,j_i)}> \dim\Gamma_{M(i,\,j_i-1)}  
    \end{equation}
    holds for $i=1,\,2$. If we have
    \begin{align} 
         \Gamma_{ j_1^\prime,\,j_2^\prime,\,j_3,\,\dots,\,j_r}
        =\Gamma_{M(1,\,j_1^\prime)}\times \Gamma_{M(2,\,j_2^\prime)} \times  \Gamma_{M(3,\,j_3)} \times \cdots \times \Gamma_{M(r,\,j_r)}
    \end{align}    
    for $(j_1^\prime,\,j_2^\prime)=(j_1-1,\,j_2),\,(j_1,\,j_2-1)$, and $(j_1-1,\,j_2-1)$, then we also have 
    \begin{equation}
    \Gamma_{ j_1,\,j_2,\,\dots,\,j_r}=\Gamma_{M(1,\,j_1)}\times \Gamma_{M(2,\,j_2)} \times \cdots \times \Gamma_{M(r,\,j_r)}. 
\end{equation}
\end{Proposition}

\begin{proof}
    Lemmas \ref{MappedOnto} and \ref{Lemmaaction} enable us to apply Lemma \ref{LemmaWisWhole}, hence the algebraic subgroup $\mathcal{V}$ of $V_{1,\,j_1}\times V_{2,\,j_2}$ is equal to the whole algebraic group $V_{1,\,j_1}\times V_{2,\,j_2}$. 
Therefore we can also obtain that $\dim \mathcal{U}= \dim (U_{1,\,j_1}\times \cdots \times U_{r,\,j_r})$ by Diagram \eqref{upperrectangle} of exact sequences and obtain $\dim \Gamma_{ j_1,\,j_2,\,\dots,\,j_r}=\dim (\Gamma_{M(1,\,j_1)}\times  \cdots \times \Gamma_{M(r,\,j_r)})$ by the commutative diagram 
            \begin{equation}    \begin{tikzpicture}[auto] 
                   \node (22) at (-1, 2) {$\mathcal{U}$}; \node (42) at (3.5, 2) {$\Gamma_{ j_1,\,j_2,\,\dots,\,j_r}$}; \node (62) at (8, 2) {$\Gamma_{ 0,\,0,\,\dots,\,0}$};
                    \node (20) at (-1, 0) {$U_{1,\,j_1}\times \cdots \times U_{r,\,j_r}$}; \node (40) at (3.5, 0) {$\Gamma_{ M(1,\,j_1)}\times\cdots \times \Gamma_{M(r,\,j_r)}$}; \node (60) at (8, 0) {$\Gamma_{ M(1,\,0)}\times\cdots \times \Gamma_{M(r,\,0)}$};
    
                    \draw[-] (62) to node[sloped] {$=$}(60);
                    \draw[{Hooks[right]}->] (42) to node {}(40);
                    \draw[{Hooks[right]}->] (22) to node {}(20);
                    
                    \draw[{Hooks[right]}->] (22) to node {}(42);
                    \draw[{Hooks[right]}->] (20) to node {}(40);
    
                    \draw[->>] (42) to node {}(62);
                    \draw[->>] (40) to node {}(60);
    
    
                \end{tikzpicture}
            \end{equation}
whose horizontal lines are exact.
As the direct product $\Gamma_{M(1,\,j_1)} \times \cdots \times \Gamma_{M(r,\,j_r)}$ is smooth and connected, we can show that the closed immersion $$\Gamma_{j_1,\,j_2,\,\dots,\,j_r} \hookrightarrow\Gamma_{M(1,\,j_1)}\times  \cdots \times \Gamma_{M(r,\,j_r)}$$
given in Lemma \ref{Lemmainclusion} is an isomorphism.
\end{proof}

\subsubsection{Proof of Theorem \ref{GaloisGroupDirectProd} and consequences on transcendence of MZV's}

We accomplish the proof of Theorem \ref{GaloisGroupDirectProd}. We also consider the transcendence result (Theorem \ref{main}), which comes from the theorem.

\begin{proof}[Proof of Theorems \ref{GaloisGroupDirectProd}]
    Propositions \ref{GaloisGroupDirectProd0}, \ref{Proptrivialcase}, \ref{GaloisGroupDirectProd2}, and \ref{GaloisGroupDirectProd3} enable us to obtain
    \begin{equation}
    \Gamma_{j_1,\,j_2,\,\dots,\,j_r}=\Gamma_{M(1,\,j_1)}\times \Gamma_{M(2,\,j_2)} \times \cdots \times \Gamma_{M(r,\,j_r)} 
\end{equation}
by induction on the sum $j_1+\cdots+j_r$.
\end{proof}

Theorems \ref{ThmPapanikolas} and \ref{GaloisGroupDirectProd} induce Theorem \ref{main} as follows:

\begin{proof}[Proof of Theorem \ref{main}]
By Theorems \ref{GaloisGroupDirectProd} and \ref{ThmPapanikolas}, we have
\begin{align}
    &\quad \ \ \trdeg_{\overline{K}}\overline{K}(\mathcal{L}_{l_i,\,\mathbf{s}_j}|_{t=\theta},\,\Omega_{l_i}|_{t=\theta} \ \mid 1\leq i \leq r,\,1 \leq j \leq \#I )\\
    &=\dim\Gamma_{M(1,\,\#I)\oplus \cdots \oplus M(r,\,\#I)}
    =\dim \Gamma_{M(1,\,\#I)}\times \Gamma_{M(2,\,\#I)} \times \cdots \times \Gamma_{M(r,\,\#I)}\\
    &=\dim \Gamma_{M(1,\,\#I)}+\dim \Gamma_{M(2,\,\#I)}+\cdots+ \dim \Gamma_{M(r,\,\#I)}\\
    &=\sum_{i=1}^{r}\trdeg_{\overline{K}}\overline{K}(\mathcal{L}_{l_i,\,\mathbf{s}_j}|_{t=\theta},\,\Omega_{l_i}|_{t=\theta} \ \mid 1 \leq j \leq \#I )
\end{align}
for given $u_{i,\,s_j}$'s which satisfy the inequality 
\begin{equation}
    ||u_{i,\,s}||_\infty<|\theta|_\infty^{\frac{s p^{l_i}}{p^{l_i}-1}}.
\end{equation}
Putting $u_{i,\,s}=H_{l_i,\,s-1}$ for each $1\leq i\leq r$ and $s\geq1$, we obtain the desired result.

\end{proof}

For any finite set $I$ of indices, there exists a finite set $\hat{I}$ which contains $I$ and satisfies the assumption in Theorem \ref{main}. Therefore, we can conclude that no algebraic relations exist over $\overline{K}$ which relate MZV's with different constants fields.
Theorem \ref{GaloisGroupDirectProd} gives us the direct product decomposition in Equation \eqref{mainthmeq} even if we do not know the precise structure of each factor $\Gamma_{M(i,\,j_i)}$.
In other words, we can obtain Equation \eqref{EqmainA} without knowing the transcendence degrees of the field 
$\overline{K}(\tilde{\pi}_{l_i},\,\zeta_{A_{l_i}}(\mathbf{s}) \mid \mathbf{s}\in I )$ over $\overline{K}$ for $1 \leq i \leq r$. 
It seems difficult to determine these transcendence degrees for a given finite set $I$ of indices.


\end{document}